\documentclass{amsart}
\usepackage[utf8]{inputenc}
\usepackage[english]{babel}
\usepackage{amsmath,amsfonts,amssymb,relsize,amsthm}
\usepackage{dsfont}
\usepackage{esint}

\usepackage{graphicx}
\usepackage{subfig}
\usepackage{tikz}
\usetikzlibrary{matrix}
\usetikzlibrary{arrows,positioning}
\usepackage{color}

\usetikzlibrary{3d}
\usepackage{xifthen}
\usetikzlibrary{decorations.pathmorphing}

\usepackage{hyperref} 

\usepackage[active]{srcltx} 
\usepackage{color,soul} 

\usepackage[backend=bibtex]{biblatex}
\newtheorem{theorem}{Theorem}[section]

\newtheorem{lemma}[theorem]{Lemma}
\newtheorem{proposition}[theorem]{Proposition}

\theoremstyle{definition}
\newtheorem{definition}[theorem]{Definition}
\newtheorem{remark}{Remark}
\newtheorem{assumption}{Assumption}

\numberwithin{equation}{section}

\def\A{{\mathcal{A}}}

\def\f{{\mathcal{F}}}
\def\g{{\mathcal{G}}}

\def\I{{\mathcal{I}}}

\def\s{{\mathcal{S}}}
\def\L{{\mathcal{L}}}
\def\M{{\mathcal{M}}}
\def\D{{\mathcal{D}}}
\def\V{{\mathcal{V}}}
\def\X{{\mathcal{X}}}

\def\R{{\mathbb{R}}}
\def\Rd{{\mathbb{R}^d}}
\def\C{{\mathcal{C}}}

\def\N{{\mathbb{N}}}
\def\P{{\mathbb{P}}}

\def\eps{{\varepsilon}}
\def\nuvec{{\boldsymbol{\nu}}}
\def\nuv{{\nu_v}}
\def\nuu{{\nu_u}}
\def\nuc{{\nu_c}}
\def\phivec{{\boldsymbol{\phi}}}

\def\phiv{{\phi_{v}}}
\def\phiu{{\phi_{u}}}
\def\phic{{\phi_{c}}}

\def\Lu{{\L^{u}}}
\def\Lc{{\L^{c}}}
\def\xiv{{\xi_{v}}}
\def\xiu{{\xi_{u}}}
\def\xic{{\xi_{c}}}

\def\Vp{{\V_{p}}}
\def\Vc{{\V_{c}}}

\def\E{{\mathbb{E}}}

\newcommand*{\myprime}{^{\prime}\mkern-1.2mu}
\newcommand*{\mydprime}{^{\prime\prime}\mkern-1.2mu}

\DeclareMathOperator{\Supp}{Supp}

\newcommand{\indic}[1]{\mathds{1}_{\left\lbrace #1 \right\rbrace}}

\newcommand{\nn}{\nonumber}

\newcommand{\closure}[2][3]{%
{}\mkern#1mu\overline{\mkern-#1mu#2}}

\newcommand\norm[1]{\left\lVert#1\right\rVert}

\title{{\bf A measure-valued stochastic model for vector-borne viruses}}
\author[Mario Ayala]{Mario Ayala}
\address{UR 546 Biostatistique et Processus  Spatiaux, INRAE, Domaine St Paul Site Agroparc, F-84000 Avignon, France.}
\email{mario-antonio.ayala-valenzuela@inrae.fr}
\author[Jerome Coville]{Jerome Coville}
\address{UR 546 Biostatistique et Processus  Spatiaux, INRAE, Domaine St Paul Site Agroparc, F-84000 Avignon, France.}
\email{jerome.coville@inrae.fr}
\author[Raphael Forien]{Raphael Forien}
\address{UR 546 Biostatistique et Processus  Spatiaux, INRAE, Domaine St Paul Site Agroparc, F-84000 Avignon, France.}
\email{raphael.forien@inrae.fr}
\begin{document}

\maketitle

\begin{abstract}
In this work we propose a measure-valued stochastic process representing the dynamics of a virus population, structured by phenotypic traits and geographical space, and where viruses are transported between spatial locations by mechanical vectors. As a first example of the use of this model, we show how to use this model to infer results on the probability of extinction of the virus population. Later, by combining various scalings on population sizes, speed of diffusion of vectors, and other relevant model parameters, we show the emergence of  two systems of integro-differential equations as Macroscopic descriptions of the system. Under the existence of densities at time zero, we also show the propagation of this property for later times, and derive the strong formulation of the limiting systems of IDEs. These strong formulations, in a sense, correspond to spatial Lotka-Volterra competition models with mutation and vector-borne dispersal.
\end{abstract}

\section{Introduction}\label{IntroSect}
The selective pressure imposed on pathogens during the infection of a new type of  host is one of the  well known mechanisms that drive their evolution. 
For plant pathogens,  plant genes conferring major or partial resistance to these pathogens are then  valuable natural resources which in a context of a plant of agricultural interest  should be used in a way that maximizes the preservation of their efficiency in conjunction with the  gain of crop productivity they can ensure.    
To gain some insights on the possible  optimized deployment strategies, the study  of adaptation of pathogens to their hosts during repeated epidemic events is then of great importance. The understanding of such evolution process is at the heart of  evolutionary epidemiology research, see for example the review \cite{Restif2009}.  
In the evolutionary epidemiology literature, most of the models deal with either the adaptation process with oversimplified epidemic process, as described in \cite{crow_introduction_19710},  or conversely focus on the epidemic process forgetting a proper description of the evolution processes at stake. One of the main reasons essentially comes from the modelling tools used to describe these two phenomena. Indeed,  selection mutation population  models  are often used to describe  evolution phenomena ignoring the possible spatial structure of this epidemic event,  whereas epidemics tend to be represented by SEIR type models using the host as proxy for describing the epidemic  forgetting the description of the evolving pathogen population.  

Modelling approaches trying to  reconcile these two points of view have been recently introduced \cite{Day2004,day_applying_2007,Ohtsuki2006,iacono_evolution_2012, djidjou-demasse_steady_2017,fabre_epi-evolutionary_2022} where the epidemic approach considering the  host as a  proxy  has been adapted in order to take into account a partial description of the population that evolves and its  possible adaptation by mutation. However, whereas this approach seems well suited  for pathogens that behave like spores, it does not seems flexible enough to describe viral populations that disseminate in a field through vectors. As well, they do not take into account stochastic effects due to low population densities at the beginning of the epidemic events.

In the spirit of \cite{Day2004,day_applying_2007,Ohtsuki2006,iacono_evolution_2012, djidjou-demasse_steady_2017,fabre_epi-evolutionary_2022},  we propose and analyse  here new modelling tools that integrate on the same time scale the epidemic events and the adaptation processes.  However, unlike the approach proposed in \cite{Ohtsuki2006,iacono_evolution_2012, djidjou-demasse_steady_2017}  that uses the hosts as a proxy,  we fully describe the pathogen population with all its possible interactions making the assumption that the hosts is an environmental variable for the pathogen population. Our aim here is then to construct a stochastic realistic representation of the main processes of adaptation involved during an epidemic event, and to obtain  various large population asymptotic limits depending on the choice of the scaling parameter considered. The stochastic nature of our model gives room to incorporate demographic stochasticity, which is an important feature to consider, since early stages potentially affect later stages of the epidemic.  Moreover, our model permits to study both quantitative and qualitative trait scenarios. 

Some of the results we present here are in the spirit of the so-called hydrodynamic limits from the theory of interacting particle systems, for physics see for example \cite{demasi2006mathematical,kipnis1998scaling,seppalainen2008translation}, and for biology \cite{fournier_microscopic_2004,champagnat_invasion_2007, champagnat_individual_2008,meleard_stochastic_2015}. In particular, we use the methodology introduced in \cite{fournier_microscopic_2004}, and in subsequent works (\cite{champagnat_invasion_2007, champagnat_individual_2008,meleard_stochastic_2015}), to build a measure-valued stochastic model representing the dynamics of a virus population structured by time, space, and  phenotypic traits, and where viruses are subject to vector-borne dispersal, and can only grow in plants.

In the derivation of the scaling limits, our particular choice of re-scaling is important since it allows us to consider different scenarios that depend on the relation of the order of magnitude between the total number of viruses and vectors at time zero. In particular, when we have fewer vectors than viruses, in order to obtain a sensible limit, it becomes necessary to accelerate the diffusion of vectors. As the reader will see in Theorem \ref{Theoremlessintro} below, this acceleration has the consequence that for each $t>0$ the  population of vectors is given by the solution of an elliptic system which depends only on the current population of viruses at time $t$. We can interpret this as saying that the population of vectors is at an equilibrium that only changes through the time evolution of the virus population.

\subsection*{The Viral Epidemic model~}~\\


Let us consider three populations, denoted by   $\nuv(t), \nuc(t)$ and $\nuu(t)$, representing the viral population, the population of vectors that are charged with a virus, and the free vectors population, respectively. As in  \cite{fournier_microscopic_2004,champagnat_invasion_2007, champagnat_individual_2008,meleard_stochastic_2015}, let us first introduce, for a Polish space $X$, the set $\mathcal{M}_p(X)$ denoting the space of finite point measures on $X$. We define these populations by means of  point measures on the adequate  Polish spaces as follows. 
First, we consider the viral population $\nuv(t)$. We represent this population  by means of  a point measure over the space $ E \times \mathcal{X} $, where $E$ is a finite set of positions in a bounded smooth, at least $C^3$, domain $\D \subset \Rd$ and $ \mathcal{X} \subseteq \R^{n} $ is a compact set. The set $ E $ corresponds to the locations of the various host plants, and $ \mathcal{X} $ corresponds to the phenotypic trait space of viruses.
Thus, for $ t \geq 0 $,
\begin{equation}
    \nuv(t) = \sum_{i=1}^{N_v(t)} \delta_{x_i(t),z_i(t)},
\end{equation}
where $ N_v(t) := \langle \nuv(t), 1 \rangle $ is the total size of the virus population and
\begin{align*}
	(x_1(t), \ldots, x_{N_v(t)}(t)) \in E^{N_v(t)}, && (z_1(t), \ldots, z_{N_v(t)}(t)) \in \mathcal{X}^{N_v(t)}
\end{align*}
correspond respectively to an arbitrary ordering of the positions and traits of the viruses.


Similarly, the uncharged vectors population is represented by a finite point measure on $ \D $, $ \nuu(t) $, taking the form
\begin{align*}
	\nuu(t) = \sum_{i=1}^{N_u(t)} \delta_{Y_i^u(t)},
\end{align*}
where $ N_u(t) $ is the number of uncharged vectors and $ (Y^u_1(t), \ldots, Y^u_{N_u(t)}(t)) $ are their positions at time $ t \geq 0 $.
Finally, the population of vectors carrying viruses is represented by a measure $ \nuc(t) \in \mathcal{M}_p(\D \times \mathcal{X}) $, of the form
\begin{align*}
	\nuc(t) = \sum_{i=1}^{N_c(t)} \delta_{Y_i^c(t), z_i(t)},
\end{align*}
where, for $ i \in \lbrace 1, \ldots, N_c(t) \rbrace $, $ Y_i^c(t) $ is the position of the vector and $ z_i(t) $ is the trait of the unique virus it is carrying.


Figure \ref{enviFig} below, shows an example of the type of environment we have in mind. We have plotted plants in different colors to convey the idea that incorporating spatial dependency of some of the rates in this framework allows us to consider different varieties of plants. Notice that it is also possible to consider non-homogeneous (spatial) distributions of plants.

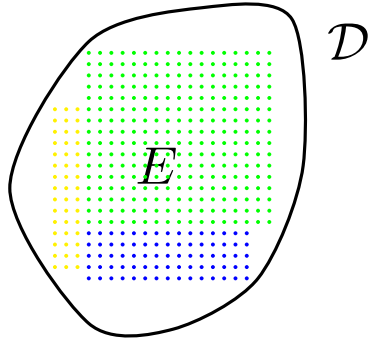
\begin{figure}[h]
\centering
\begin{tikzpicture} [line width=1.2, scale=1.5]
\draw plot [smooth cycle] coordinates {(1.0,.1)(1.75,0.05)(2.5,.5)(2.9,1.5)(2.8,2.8)(2.0, 2.9)(1.0,2.6)(0.3, 1.3)} node at (3.3,2.6)[ scale=2] {$\mathcal{D}$};
\draw node at (1.6,1.5)[scale=2]{$E$};
 \foreach \x in {1,1.1,...,2.7}
    \foreach \y in {1,1.1,...,2.6}
    {
    \fill[color=green] (\x,\y) circle (0.5pt);
    }
  \foreach \x in {1,1.1,...,2.5}
    \foreach \y in {0.5,0.6,...,1}
    {
    \fill[color=blue] (\x,\y) circle (0.5pt);
    }
  \foreach \x in {0.7,0.8,...,1}
    \foreach \y in {0.6,0.7,...,2.0}
    {
    \fill[color=yellow] (\x,\y) circle (0.5pt);
    }
\end{tikzpicture}
\caption{The spatial domain $\D$ represents the space in which vectors can diffuse (with normal reflection at the boundary). The colored dots represent the set $E$, i.e. the location of plants}
\label{enviFig}
\end{figure}

In addition to the definitions of the point measures $\nuv(t), \nuc(t)$ and $\nuu(t)$, let us now define a specific set of biological rules that mechanistically determine the way that these measures evolve over time. To reflect the adaptation process during an epidemic event, we will assume that viruses are submitted to growth, death, mutation  and transportation by vectors and that vectors move in $\D$ according to a reflected It\^o-diffusion and  that they can load and unload viruses on plants.  This translates into considering the following processes :

\begin{description}
\item[Reproduction] viruses on plants reproduce asexually at rate $b(x,z)$, where $x \in E$ denotes the position of the hosting plant and $z \in \X \subset \R^n$ the phenotypic trait of the virus. Given a reproductive event, with probability $1-\mu$ the new virus inherits its parent trait, and with probability $\mu$ undergoes mutation. In the latter case, the trait $z\myprime$ of the new individual is drawn from a probability distribution with
density $m(z,z\myprime)$ with respect to some measure $\bar{m}(dz\myprime)$ on $\X$.
\item[Death]viruses on plants die naturally at rate $d(z)>0$, and due to competition at rate $c N_x(t)$, where $N_x(t)$ denotes the number of viruses in a plant at position $x \in E$, and $c$ is a positive constant. Likewise, viruses being carried by a vector die at rate $\gamma(z)$.
\end{description}

\begin{description}
\item[Diffusion] vectors, free or charged with virus, diffuse in the spatial domain $\D \subset \Rd$ according to a
 non-degenerate (see \cite{fournier2010absolute}) It\^o diffusion with normal reflection at the boundary of the domain $\partial \D$, i.e. for $\alpha \in \lbrace u, c \rbrace$,  the diffusion has infinitesimal generator
\begin{align*}\label{Itogen}
\Lu \phiu(y) &= a^u (y) \cdot \nabla \phiu(y) + \frac{\sigma^u(y)^2}{2} \Delta \phiu(y), \\
\Lc \phic(w,e) &= a^c (w) \cdot \nabla^{(w)} \phic(w,e) + \frac{\sigma^c(w)^2}{2} \Delta^{(w)}  \phic(w,e), 
\end{align*}
where $\phiu : \bar{\D} \to \R$ and $\phic : \bar{\D} \times \X \to \R$ are elements of their corresponding domains:
\begin{align*}
D(\Lu) &= \left\{ \phi \in \C^2(\D): \nabla \phi(y) \cdot \vec{n}(y)  = 0, \quad \forall y \in \partial D \right\}, \\
D(\Lc) &= \left\{ \phi: \D \times \X \to \R : \forall e \in \X; \phi(\cdot,e) \in \C^2(\D) \text{ and }  \nabla \phi(y) \cdot \vec{n}(y)  = 0, \quad \forall w \in \partial D \right\}.
\end{align*}
where $\vec{n}(y)$ denotes the inward normal at $y \in \partial D$.
\item[Charging of vectors] vectors \textit{successfully} bite plants and get charged with a virus at rate $\beta(t,y,x,N_x(t),z)$, where $y \in \D$ denotes the current position of the vector, $x \in E$  the plant's location, $N_x(t)$ is the total number of viruses at $x$, and $z \in \X$ is the phenotype of the virus being taken. Trying to be consistent with biological studies (see \cite{moury2007estimation,gutierrez2012virus}) which suggest that a vector  carries and effectively transmits a small number of viruses, we assume that a vector can carry at most one virus at the same time.  
\item[Un-loading of vectors] vectors \textit{successfully} discharge viruses at rate $\eta(t,y,x,z)$, where $y \in \D$ denotes the current position of the vector, $x \in E$ the plant's location, and $z \in \X$ the phenotype of the virus being unloaded. 
\end{description}

Furthermore, we also make the following, biologically consistent, additional assumptions:
\begin{assumption}\label{AssPherates}
Let us assume that the coefficients $\sigma, a $ are Lipschitz continuous, $\sigma >0$, that the mutation density is a probability density, that the competition parameter $c$ is strictly positive, and  
that there exist $(\bar{b},\bar{d},\bar{\gamma}, \bar{\beta},\bar{\eta}) \in \R_+^5$, such that for all  $t \geq 0$ we have:
\begin{align*}
 \sup_{ (x,z) \in E \times \X}  b(x, z)  = \bar{b} < \infty, &\quad \sup_{ z \in \X }d(z) = \bar{d} < \infty, \quad    \sup_{ u \in \X}    \gamma(u)  =  \bar{\gamma} < \infty  \nonumber \\
 \sup_{r \in \R_+} \sup_{(y,z) \in \D \times \X } \int_E \beta(t,y,x,r,z) \, dx &= \bar{\beta} < \infty, \quad \sup_{ (y,z) \in \D \times \X } \int_E \eta(t,y,x,z) \, dx = \bar{\eta} < \infty.
\end{align*}
Moreover, for all $(t,y,x,z) \in \R_+ \times \D \times E \times \X$, we assume $\beta(t,y,x,\cdot,z): \R_+ \to \R_+$ to be Lipschitz continuous. 
\end{assumption}

Our framework allows to incorporate the spatial extent of the plant. We can assume for example the existence of some function $\beta_e : \R_+ \times \D \times E \times \R \times \X \to \R_+$, and $r_p>0$ such that:
\begin{equation*}
\beta(t,y,x, N_x(t),z) = 
\begin{cases}
\beta_e(t,|x-y|,N_x(t), z) & \text{ if } |x-y| \leq r_p, \\
0 & \text{ otherwise},
\end{cases}
\end{equation*}
for all $t\geq 0$. Analogous assumptions can be imposed over the model parameter $\eta$.

\subsection*{Infinitesimal generator}
We are interested in the time-evolution of the process $\lbrace \nuvec_t : t \geq 0 \rbrace= \lbrace (\nuv(t),\nuu(t),\nuc(t)) : t \geq 0 \rbrace$ taking values in the space of measures $\M_p := \M_p( E \times \X) \times \M_p(\D) \times \M_p( \D \times \X) $.  Let us introduce the set of cylindrical functions that generates the set of bounded and measurable functions from $\M_p$ to $\R$, necessary to describe the generator of the process. 

\begin{definition}\label{correncod}
We call an admissible triplet, and denote it by $\phivec$, a set of bounded and measurable functions $ \lbrace \phiv, \phiu, \phic \rbrace $ such that: $\phiu: \D \to \R$ is in the domain of the generator $\Lu$, $\phic: \Vc \to \R$ is in the domain of the generator $\Lc$. Moreover, we denote by $\Phi(\M_p)$ the set of all admissible triplets of this form.
\end{definition}

\begin{remark}\label{DiffuEnco}
For $\phivec \in \Phi(\M_p)$, we denote by $\L^\alpha \phivec$, the following admissible triplets:
\begin{align}
\L^u \phivec = \lbrace 0, \L^u_{y} \phiu, 0 \rbrace,\text{ and } \L^c \phivec = \lbrace 0,0, \L^c_{w} \phic \rbrace. \nonumber
\end{align}
\end{remark}
We now define the relevant set of cylindrical functions.
\begin{definition}\label{CyliDef}
The class of cylindrical functions $\f_C$ on $\M_p$ is given by functions $F_\phivec: \M_p \to \R$, of the form:
\begin{align*}
    F_{\phivec}(\nuvec) &:= F(\langle \phivec, \nuvec \rangle) := F (\langle\phiv,\nuv\rangle,\langle \phiu, \nuu\rangle,\langle \phic, \nuc\rangle)
\end{align*}
where $F \in \C^2(\R^3;\R)$, and $\phivec=\lbrace \phiv, \phiu, \phic \rbrace$ is an admissible triplet, and for $\alpha \in \lbrace v,u,c \rbrace$ we abused notation by using:
\begin{equation*}
\langle \phi_\alpha, \nu_\alpha \rangle :=  \int_{\V_{\alpha}} \phi_\alpha(x) \, \nu_\alpha(dx), 
\end{equation*}
with $\V_v= E \times \X, \V_u = \D$, and $\V_c= \D \times \X$.
\end{definition}

\begin{remark}\label{remarkusefunction0}
By Remark 1.1 in \cite{champagnat_invasion_2007} we know that (the spaces) $\C^{2}_0(\D)$ and $\C^{2,0}_0(\Vc)$ are dense in $\C(\D)$ and $\C(\Vc)$ for the uniform topology.
\end{remark}

As we did before, and whenever possible in the future, we will be consistent with our notation. Whenever we use the functions $\phivec=\lbrace  \phiv, \phiu, \phic \rbrace$ we will refer to functions satisfying the conditions of this section.

The infinitesimal generator $\L$, that corresponds to the dynamics described above, can be seen as the sum of a jump part, denoted by $\L_1$, and a diffusive part that we denote by $\L_2$. We  further split the jump part of the generator as the sum of operators, acting on cylindrical functions $F_{\phivec} \in \f_C$, dealing with every type of jump event:\\

The operator concerning demographics of viruses in plants is given as follows:
\begin{align*}
 \L^{dem} F_{\phivec}(\nuvec) &:=  (1-\mu) \int_{\Vp} b(x,z) \left[F_{\phivec}( (\nuv + \delta_{x,z}, \nuu, \nuc)) -F_{\phivec}(\nuvec) \right] \nuv(dx,dz) \nonumber \\
  &+ \mu \int_{\Vp} b(x,z) \int_{\X} m(z,e) \, \left[F_{\phivec}( (\nuv + \delta_{x,e}, \nuu, \nuc)) -F_{\phivec}(\nuvec) \right] \nuv(dx,dz) \, \bar{m}(de) \nonumber \\
  &+ \int_{\Vp} (d(z)+ c \langle \nu_v^{x},1 \rangle ) \, \left[F_{\phivec}( (\nuv - \delta_{x,z}, \nuu, \nuc)) -F_{\phivec}(\nuvec) \right] \nuv(dx,dz). 
\end{align*}
where for $ x \in E$, the measure $\nu_v^x$ denotes the restriction of $\nu_v$ to  $\lbrace x \rbrace \times \X$.

The loading and unloading of viruses on a vector is described in terms of the following operators:
 \begin{align*}
 \L^{load} F_{\phivec}(\nuvec) &:=  \L^{load} F_{\phivec}((\nuv,\nuu,\nuc)) \nonumber \\
  &= \int_{\Vp }\int_{D} \beta(t,y,x,N_x(t),z) \, \left[F_{\phivec}( (\nuv - \delta_{x,z}, \nuu-\delta_y, \nuc+\delta_{y,z})) -F_{\phivec}(\nuvec) \right] \nuv(dx,dz) \, \nuu(dy),
\end{align*}   
and 
 \begin{align*}
 \L^{unload} F_{\phivec}(\nuvec) &:=  \L^{unload} F_{\phivec}((\nuv,\nuu,\nuc)) \nonumber \\
  &= \int_{\Vc} \int_{E} \eta(t,w,x,e) \, \left[F_{\phivec}( (\nuv + \delta_{x,e}, \nuu+\delta_w, \nuc-\delta_{w,e})) -F_{\phivec}(\nuvec) \right] \nuc(dw,de) \, dx. 
\end{align*}   
The final type of jump event, the death of viruses on vectors, is given by:
 \begin{align*}\label{DefLlos}
 \L^{los} F_{\phivec}(\nuvec) &:=  \L^{los} F_{\phivec}((\nuv,\nuu,\nuc)) \nonumber \\
  &= \int_{\Vc} \gamma(e) \, \left[F_{\phivec}( (\nuv, \nuu+\delta_w, \nuc-\delta_{w,e})) -F_{\phivec}(\nuvec) \right] \nuc(dw,de). 
\end{align*} 
Summing up we have that the jump part is given by:
\begin{equation*}\label{L1def}
    \L_1 := \L^{dem}+ \L^{load} + \L^{un-load}+\L^{los}
\end{equation*}

For the diffusive part we need an admissible triplet $\phivec=\lbrace \phiv, \phiu, \phic \rbrace$ to be as in Definition \ref{correncod}. For functions $F_\phivec \in \f_C$ the diffusive part $\L_2$ can be obtained from It\^o's formula. The $\L_2$ operator is given by:
\begin{align}\label{DefgenL2}
&\L_2 F_\phivec(\nuvec) \nonumber \\ &:=\left(\int_{\D} \L_y^u \phiu(y) \nuu(dy)+\int_{\Vc} \Lc \phic(w,e) \nuc(dw,de) \right) \,  F\myprime_\phivec(\nuvec) \nonumber \\
&+\left( \int_{\D} \frac{(\sigma^u(y))^2}{2} |\nabla_y \phiu(y) |^2  \nuu(dy) +\int_{\Vc} \frac{(\sigma^c(w))^2}{2} |\nabla_w \phic(w,e) |^2  \nuc(dw,de) \right)\ F\mydprime_\phivec(\nuvec)
\end{align}
where we have abused notation by using $F_{\phivec}^{\prime}$ and $F_{\phivec}\mydprime$, and hence neglect the possible lack of membership of the functions $F\myprime$ and $F\mydprime$ to $\C^2(\R^3)$.

After introducing all the operators we have:
\begin{equation}\label{DefGen}
    \L F_\phivec(\nuvec) = \L_1 F_\phivec(\nuvec) + \L_2 F_\phivec(\nuvec)
\end{equation}
The description of our processes in terms of its infinitesimal generator is rather formal. We refer the reader to the Appendix \ref{SecWell} of this work, were in the vein of \cite{fournier_microscopic_2004}, we provide a rigorous definition on path-space, and a proof of the well-definedness of the processes corresponding to the generator $\L$.

\begin{remark}\label{Conservvectors}
Notice that the dynamics described above leaves invariant the total number of vectors.
\end{remark}

\subsection*{Existence and uniqueness}
We have just described the time-evolution of our measure-valued process in terms of the infinitesimal generator $\L$ given by \eqref{DefGen}. 
Our first result is to rigorously prove that under Assumption \ref{AssPherates} the process is well defined.

\begin{theorem}\label{Theoremwelldefined}
Let $\nuvec_0 = (\nuv(0), \nuu(0),\nuc(0))$ be such that we have:
\begin{equation}\label{pboundt02}
  \mathbb{E} \left[  \langle \mathbf{1}, \nuvec_0 \rangle^2 \right] < \infty ,    
\end{equation}
where
\begin{equation*}
\langle \mathbf{1}, \nuvec_0 \rangle := \int_{E \times \X} 1 \, \nuv(0)(dx,dz) +\int_{\D} 1 \, \nuu(0)(dy) + \int_{\D \times \X} 1 \, \nuc(0)(dy,dz).     
\end{equation*}
Then, under Assumption \ref{AssPherates}, part a), the process $\left(\nuvec_t\right)_{t \geq 0} =\left( \nu_v(t),\nuu(t), \nuc(t) : t \geq 0  \right)$ satisfies:
\begin{equation}\label{pboundtone}
  \mathbb{E} \left[  \sup_{t \in [0,T]}\langle \mathbf{1}, \nuvec_t \rangle^2  \right] < \infty.     
\end{equation}
In particular, we also have that the process $\nu_t$ is well-defined.
\end{theorem}
For a proof we refer to the Appendix, where we proved \eqref{pboundtone} for higher-order moments, and where we also showed that the dynamics of the process $( \nuvec_t)_{t \geq 0}$ indeed corresponds to the one described by the infinitesimal generator $\L$ .

\subsection*{Population-level descriptions as observables}
The framework of this paper allows us to directly recover some quantities of interest as observables of the system. For example, the evolution of the process representing the total virus population can be recovered by integrating the constant function 1 with respect to the measures $\nuv$ and $\nuc$:
\begin{equation}\label{TotalP}
    P_v(t) := \langle 1, \nuv(t)  \rangle + \langle 1, \nuc(t)\rangle.
\end{equation}
In a similar way, we can also recover the total number of vectors charged with a virus and those free from viruses:
\begin{align}
    N_u(t) := \langle 1, \nuu(t) \rangle,  \text{ and } N_c(t) := \langle 1, \nuc(t) \rangle. \nonumber
\end{align}
Moreover, we show how this description can be used to extend known results about extinction to the case of vector-borne dispersal.

\begin{theorem}\label{ExtincTheoremintro}
Suppose that the initial set of measures $(\nuv(0), \nuu(0), \nuc(0))$ is such that
\begin{equation*}
    \E \left[ \langle \nu_v(0) , 1 \rangle \right] < \infty.
\end{equation*}
Under Assumption \ref{AssPherates}, we have that the total-virus population process $P_v(t)$ given by \eqref{TotalP} goes extinct almost surely.
\end{theorem}

We refer to Section \ref{ExtinctionSect} for details on the proof of this theorem.\\

\subsection*{Deterministc limits}
Our main results concern the derivation of large population-level deterministic descriptions of our system. By introducing a rescaling parameter $K$, we derive these descriptions in the spirit of a law of large numbers for our processes. In this setting, the rescaling parameter $K$ has the biological interpretation of imposing a carrying capacity to the system. We incorporate this idea by letting the competition parameter $c$ depends on $K$ as follows:
\begin{equation*}
    c_K = \frac{c}{K}.
\end{equation*}
Moreover, we consider two scenarios: one in which the population of viruses and vectors are of the same order, and one in which the population of vectors is of smaller order than that of viruses. We model these two scenarios by introducing a new parameter $\lambda \in (0,1]$, which rescales the population of viruses and vectors, at time zero, as follows:
\begin{align*}\label{lambdak2zero}
    \nuv^{(K)}(0) = \frac{1}{K} \nuv(0,K), \quad \nuu^{(K)}(0) = \frac{1}{K^\lambda} \nuu(0,K), \text{ and } \nuc^{(K)}(0) = \frac{1}{K^\lambda} \nuc(0,K).  
\end{align*}
where for each $\alpha \in \lbrace v,u, c \rbrace$ and $K \in \N$, the random measure $\nu_\alpha(0,K)$ is an element of $\M_p(\V_\alpha)$. An important ingredient needed for our results is the assumption that at time zero the sequences  $\lbrace \nu_\alpha^{(K)}(0) \rbrace_{ K \in \N}$ converge for $\alpha \in \lbrace v, u , c \rbrace$. The convergence of these sequences implies that for the case $\lambda \in (0,1)$, the total number of vectors is of smaller order than the total population of viruses. This \textit{lack of vectors} suggests the need to let the rest of the parameters depend on $K$ as well. In particular, grossly speaking,  we compensate the lack of vectors by letting the processes $\nuu(t,K)$ and $\nuc(t,K)$ (i.e. the time evolution of the processes with initial condition $\nuu(0,K)$ and $\nuc(0,K)$) evolve with diffusive generator given by: 
\begin{equation}\label{lambdaKaccel}
    \L_K^{\alpha,\text{accel}} = K^{1-\lambda} \L^\alpha, 
\end{equation}
where for $\alpha \in \lbrace u, c \rbrace$, $\L^\alpha$ denotes the infinitesimal generator of the It\^o diffusion driving the movement of the $\alpha$-class of vectors. Furthermore, the explicit form on which the rescaled parameters $\eta_K$, $\beta_K$ and $\gamma_K$ will depend on $K$ comes from the idea of making the distance travelled by a vector between loading and unloading of order one. Additionally, we wish the number of virus deaths (on vectors) to be of the same order. This means:
\begin{equation*}
    K^{1+\lambda} \beta_K = K^\lambda \eta_K = K^\lambda \gamma_K = O(K)
\end{equation*} 
or equivalenty
\begin{equation*}
 \eta_K = O(K^{1-\lambda}), \quad 
 \gamma_K = O(K^{1-\lambda}),    
\text { and }
 \beta_K = O(K^{-\lambda}).  
\end{equation*}
We formalize these ideas with the following assumption:
\begin{assumption}\label{Assbetakreg2}
There exists a Lipschitz continuous function $\beta:\R_+ \times \mathcal{D} \times E \times \R_+ \times \mathcal{X} \to \R$, continuous functions $\eta:\R_+ \times \mathcal{D} \times E \times \mathcal{X} \to \R$ and $\gamma: \mathcal{X} \to \R$, and a positive constant $c$ such that:
\begin{equation*}
     \beta_K(t,y,x,N_x(t),z) = K^{-\lambda} \beta(t,y,x,N_x(t),z), \quad c_K = \frac{c}{K}, \quad      \eta_K(t,y,x,z) = K^{1-\lambda} \eta(t, y,x,z), \quad      \gamma_K(z) = K^{1-\lambda} \gamma(z),
\end{equation*}
for all $t \in \R_+$. Moreover, we assume the rest of the parameters to be independent of the scaling parameter $K$, and all parameters together satisfy Assumption \ref{AssPherates}.
\end{assumption}
Under the assumptions above, we normalize our processes as follows:
\begin{align}\label{lambdak2}
    \nuv^{(K)}(t) = \frac{1}{K} \nuv(t,K), \quad \nuu^{(K)}(t) = \frac{1}{K^\lambda} \nuu(t,K), \text{ and } \nuc^{(K)}(t) = \frac{1}{K^\lambda} \nuc(t,K).  
\end{align}
where for $\alpha \in \lbrace u, c \rbrace$, the process $\nu_\alpha(t,K)$ has dynamics with diffusive generator given by $\L_K^{\alpha,\text{accel}}$ as in \eqref{lambdaKaccel}.\\

In the two regimes described above ( i.e. if $\lambda=1$ or $\lambda \in (0,1)$) we show that, as the parameter $K$ tends to infinity, the normalized triplet of  measure-valued processes $(\nuv^{(K)}(t), \nuu^{(K)}(t), \nuc^{(K)}(t) )$ converges in path-space to a deterministic limiting triplet $(\xi_v(t), \xi_u(t), \xi_c(t) )$ characterized as being the solution of a system of non-local integro-differential equations. We refer the reader to Theorem \ref{IDElambdaequal} and Theorem \ref{IDElambdaless} in Section \ref{Mainresultssect}, which make rigorous this convergence in a general setting and describe the precise assumptions needed at time zero. In the following paragraphs, we present two simpler versions of Theorem \ref{IDElambdaequal} and Theorem \ref{IDElambdaless}  where in particular the deaths of viruses on vectors are neglected.

\subsection*{First regime}
For simplicity let us assume that at time zero the sequences of finite point measures $\lbrace \nuv^{(K)}(0) \rbrace_{ K \in \N}$, $\lbrace \nuu^{(K)}(0) \rbrace_{ K \in \N}$ and $\lbrace \nuc^{(K)}(0) \rbrace_{ K \in \N}$ converge to deterministic limiting measures $\xi_v(0)$, $\xi_u(0)$ and $\xi_c(0)$, respectively. Moreover, we assume that these limiting measures are absolutely continuous with respect to the relevant Haar measures in their respective state spaces (i.e., the counting measure $\zeta_E( dx)$ on $E$, the Lebesgue measure on $\D$, etc\ldots). Let us denote the densities of the limiting measures $\xi_v(0)$, $\xi_u(0)$ and $\xi_c(0)$, by $g_v(x,z), g_u(y)$, and $g_c(y,z)$ respectively. In Theorem \ref{PropAbsol}, we show that for reversible diffusions we have the propagation of absolutely continuity for later times. With an abuse of notation let us also denote by $g_v(t,x,z), g_u(t,y)$, and $g_c(t,y,z)$ respectively, these densities for later times $t>0$.\\

For the sake of clarity let us assume that vectors diffuse in space according to reflected Brownian motion. Then a consequence of Theorem \ref{IDElambdaequal} is the following result which resembles the so-called Lotka-Volterra competition system in the presence of diffusion by vectors. 

\begin{theorem}\label{Theoremequalintro}
Let $\lambda=1$, and let $\nuvec^{(K)}_0= (\nuv^{(K)}(0),\nuu^{(K)}(0), \nuc^{(K)}(0)) $ be such that:
\begin{equation*}
    \E \left[ \left(\langle 1, \nuv^{(k)}(0) \rangle + \langle 1, \nuu^{(k)}(0) \rangle + \langle 1, \nuc^{(k)}(0) \rangle \right)^3 \right] < \infty 
\end{equation*}
Then, the process $\lbrace \nuvec^{(K)}(t) : t \geq 0 \rbrace$ converges, as $K \to \infty$, to a deterministic process $\xi_t = (\xi_v(t), \xi_u(t), \xi_c(t))$ with densities $(g_v, g_u, g_c)$ being the unique solution of the following system:
\begin{equation*}
\left\{ 
\begin{aligned}
\frac{\partial}{\partial t} g_v(t,x,z) &=  (1-\mu) b(x,z) \, g_v(t,x,z)+ \mu \int_\X b(x,z\myprime) m(z^\prime, z) \,  g_v(t,x,z\myprime) \, dz^\prime    \nonumber \\
&- \left[ d(z) + c \, \left( \int_\X g_v(t,x,z^\prime) dz\myprime \right) \right]g_v(t,x,z)   \nonumber \\
&- \left[\int_\D \beta\left(y,x,\int_\X g_v(t,x,z^\prime) dz^\prime,z\right) g_u(t,y) \, dy \right] \, g_v(t,x,z) +\left( \int_{\D} \eta(y,x,z) \,  g_c(t,y,z) \, dy \right), \nonumber \\
\frac{\partial}{\partial t} g_u(t,y)  &=  \Delta g_u(t,y) - \left[\int_{ E \times \X} \beta\left(y,x,\int_\X g_v(t,x,z^\prime) dz^\prime,z\right) \,  g_v(t,x,z) \, \zeta_E( dx) \, dz \right]   \, g_u(t,y)  \nonumber \\
&+\int_{E \times \X} \eta(y,x,z) \,  g_c(t,y,z) \, \zeta_E( dx) \, dz, \nonumber \\
\frac{\partial}{\partial t} g_c(t,y,z)&= \Delta_y g_c(t,y,z)+ \left[\int_{ E } \beta\left(y,x,\int_\X g_v(t,x,z^\prime) dz^\prime,z\right) \,  g_v(t,x,z) \, \zeta_E( dx)  \right]   \, g_u(t,y)  \nonumber \\
&-\left(\int_{E} \eta(y,x,z) \,\zeta_E( dx) \right)\,  g_c(t,y,z), 
\end{aligned}
\right.
\end{equation*}
with initial data and boundary conditions given by:
\begin{align*}
&g_v(0,x,z) = g_v(x,z), \quad g_u(0,y) = g_u(y), \text{ and } \quad g_c(0,y,z) = g_c(y,z)  \\
&\nabla g_u(t,y) \cdot \vec{n}(y) = 0  \quad \text{ and } \quad \nabla g_u(t,y,z) \cdot \vec{n}(y) = 0 
\end{align*} 
for all $ y \in \partial \D$.\\
\end{theorem}

\begin{remark}
Notice that this system extends the L-V competition model with mutation by incorporating two non-local additional terms representing vector-borne dispersal. Moreover, the system also keeps track of the uncharged and charged vector populations simultaneously with that of the viruses. Finally, this representation allows us to consider a continuous trait space $\X \subset \R^n$. This structure can be used for example to study the the adaptation of viruses to resistant genes of plants \cite{fabre2009key, fabre2012modelling}.
\end{remark}

\subsection*{Second regime}
For the second regime, under the same assumptions on the sequence of initial measures, the following result is a consequence of Theorem \ref{IDElambdaless}. It shows how the speeding up of the diffusion of vectors has the effect of changing the limiting evolution of the two populations of vectors to an equilibrium state.

\begin{theorem}\label{Theoremlessintro}
Let $\nuvec^{(K)}_0= (\nuv^{(K)}(0),\nuu^{(K)}(0), \nuc^{(K)}(0)) $ be such that:
\begin{equation*}
    \E \left[ \left(\langle 1, \nuv^{(k)}(0) \rangle + \langle 1, \nuu^{(k)}(0) \rangle + \langle 1, \nuc^{(k)}(0) \rangle \right)^3 \right] < \infty 
\end{equation*}
Then, the process $\lbrace \nuvec^{(K)}(t) : t \geq 0 \rbrace$ converges, as $K \to \infty$, to a deterministic process $\xi_t = (\xi_v(t), \xi_u(t), \xi_c(t))$ with densities $(g_v, g_u, g_c)$ being solution of the following system:
\begin{equation}
\left\{ 
\begin{aligned}
\frac{\partial}{\partial t} g_v(t,x,z) &=  (1-\mu) b(x,z) \, g_v(t,x,z)+ \mu \int_\X b(x,z\myprime) m(z^\prime, z) \,  g_v(t,x,z\myprime) \, dz^\prime    \nonumber \\
&- \left[ d(z) + c \, \left( \int_\X g_v(t,x,z^\prime) dz\myprime \right) \right]g_v(t,x,z)   \nonumber \\
&- \left[\int_\D \beta\left(y,x,\int_\X g_v(t,x,z^\prime) dz^\prime,z\right) g_u(t,y) \, dy \right] \, g_v(t,x,z) +\left( \int_{\D} \eta(y,x,z) \,  g_c(t,y,z) \, dy \right), \nonumber \\
\Delta g_u(t,y) &- \left[\int_{ E \times \X} \beta\left(y,x,\int_\X g_v(t,x,z^\prime) dz^\prime,z\right) \,  g_v(t,x,z) \, \zeta_E( dx) \, dz \right]   \, g_u(t,y) + \int_{E \times \X} \eta(y,x,z) \,  g_c(t,y,z) \, \zeta_E( dx) \, dz = 0, \nonumber \\
 \Delta_y g_c(t,y,z) &+ \left[ \int_{ E } \beta\left(y,x,\int_\X g_v(t,x,z^\prime) dz^\prime,z\right) \,  g_v(t,x,z) \, \zeta_E( dx)  \right]   \, g_u(t,y)  -\left(\int_{E} \eta(y,x,z) \, \zeta_E( dx) \right)\,  g_c(t,y,z) = 0,
\end{aligned}
\right.
\end{equation}
with the same boundary conditions as in Theorem \ref{Theoremequalintro}.\\
\end{theorem}

\subsection*{Perspectives on extension to a continuous set of plants}
As a final remark, we want to mention that in Theorem \ref{Theoremequalintro} and Theorem \ref{Theoremlessintro} we have deliberately used the notation:
\begin{equation*}
    \int_E f(x) \, \zeta_E(dx),
\end{equation*}
instead of the perhaps more natural:
\begin{equation*}
    \sum_{x_i \in E} f(x_i).
\end{equation*}
Our intention with this choice was to suggest in particular the possible extension of our results to the case in which the set of plants becomes increasingly large together with the rescaling parameter $K$. To be more precise, when the locations of plants form and increasing sequence of lattices $E_K$ approximating the whole plantation space $\D$ (and  the measures $\zeta_{E_K}$ approximating the Lebesgue measure $dx$) as $K \to \infty$. We suspect that Theorems \ref{IDElambdaequal}
and \ref{IDElambdaless} can be extended to that setting. 

\subsection*{Local persistence for the first regime}
As a small application of Theorem \ref{Theoremequalintro}, in this section we give sufficient conditions for the local persistence of the virus population of a specific virus trait on an given plant, i.e., conditions that guarantee that:
\begin{equation*}
\liminf_{t \to \infty } g_v(t,x,z) >0, 
\end{equation*}
for $x \in E$, and $ z \in X$.\\

Let us introduce the following quantity:
\begin{equation*}
    R(x,z) = (1-\mu) b(x,z)  -  d(z) -\int_\D \beta(y,x,z)  \, dy
\end{equation*}
we then have the following:
\begin{proposition}\label{PersCond}
Let $g_v(t,\cdot,\cdot)$ be given as in Theorem \ref{Theoremequalintro} with initial condition:
\begin{equation}\label{condongvcero}
    g_v(0,x,z) >0.
\end{equation}
Assume that, for some fixed $x \in E$, and $ z \in X$ we have:
\begin{equation}\label{condonR}
    R(x,z) > 0.
\end{equation}
Then the virus population of trait $z$ locally persists at $x$.
\end{proposition}

\subsection*{Organization of the paper}
The rest of our paper is organized as follows. In Section \ref{ExtinctionSect} we provide the proof of Theorem \ref{ExtincTheoremintro}. In Section \ref{Mainresultssect} we state our main theorems, namely Theorem \ref{IDElambdaequal} and Theorem \ref{IDElambdaless}, together with the details on the derivation of Theorem \ref{Theoremequalintro}, and Theorem \ref{Theoremlessintro}, from Theorem \ref{IDElambdaequal}. Section \ref{ProofLocPer} contains the details of the derivation of Proposition \ref{PersCond}. In Section \ref{Proofs} we prove our main theorems using the standard compactness-uniqueness approach, in particular Section \ref{firstproof} deals with the details of the proof of Theorem \ref{IDElambdaequal}, while Section \ref{secondproof} deals with the proof of Theorem \ref{IDElambdaless} with the help of Kurtz' averaging principle for slow-fast systems. Finally, in the Appendix we include the rigorous definition of the processes on path-space. There, we also include the derivation of some standard martingale properties that are used in our proofs. 




\section{Extinction probabilities}\label{ExtinctionSect}
We are interested in the probability of extinction of the virus population, i.e., the probability
\begin{equation*}
    \P_{\nu(0)} ( \exists s >0,  P_v(s) = 0),
\end{equation*}
where the process is initially started from the measure $\nuvec(0)= (\nuv(0),\nuu(0),\nuc(0))$.\\

We first make the following remarks:
\begin{remark}
Notice that the random variable $P_v(t)$ does not make distinction about the different phenotypes. Moreover, it also takes into account the sub-population of viruses residing on vectors. This is needed to properly study extinction since it avoids the possibility of re-emergence of the virus population.
\end{remark}

\begin{remark}
The random process $\lbrace P_v(t) \rbrace_{t \geq 0}$ is integer valued, in fact it is $( \N \cup \lbrace 0 \rbrace)$-valued. Moreover, zero is an absorbing state, i.e, if for some $s \geq 0$, $P_v(s)=0$, then $P_v(t)=0$ for all $t \geq s$.
\end{remark}

\begin{remark}
Notice that by Remark \ref{Conservvectors} we have:
\begin{equation*}
    N_u(t) + N_c(t) =: V_0
\end{equation*}
for all $t \geq 0$.
\end{remark}


Before proving Theorem \ref{ExtincTheoremintro}, we need the following lemma:
\begin{lemma}\label{inequalarren}
Let $\nu_v \in \M_p( \Vp)$, then we have:
\begin{equation}\label{ineqecuaarren}
 \int_{\Vp} \left[  \int_{\X} 1 \cdot  \nuv^{x}(t)(dz\myprime)  \right] \nuv(t)(dx,dz) \geq \frac{1}{\lvert E\rvert} N_v(t)^2,
\end{equation}
where $\lvert E\rvert$ denotes the cardinality of the finite set $E$.
\end{lemma}

\begin{proof}
For $\nuv \in \M_p( \Vp)$, let us denote by $\bar{\nu}_v  \in \M_p(E)$ the following measure:
\begin{equation*}
    \bar{\nu}_v(dx) = \int_{X} 1 \, \nu_v(dx,dz). 
\end{equation*}
Notice then that we can rewrite the LHS of \eqref{ineqecuaarren} as follows:
\begin{equation*}
 \int_{E \times \X} \left[  \int_{\X} 1 \cdot  \nu_v^{x}(t)(dz\myprime)  \right] \nu_v(t)(dx,dz) =  \int_{E}  \int_{E} \indic{x}(x\myprime) \cdot  \bar{\nu}_v^{x}(t)(dx\myprime) \,    \bar{\nu}_v^{x}(t)(dx)
\end{equation*}
We conclude the proof of the lemma by using the fact that the set $E$ is finite, and the fundamental inequality:
\begin{equation*}
    n \sum_{i=1}^n  a_i^2 \geq \left(\sum_{i=1}^n  a_i\right)^2 
\end{equation*}
for $a_i \geq 0$.
\end{proof}

We now can proceed to the proof of Theorem \ref{ExtincTheoremintro}.

\begin{proof}
Let us first claim:
\begin{equation}\label{step1}
    \sup_{t\geq 0} \, \E \left[ P_v(t) \right] < \infty.
\end{equation}
To see that this is the case, let us define $f(t)$ as follows:
\begin{equation*}
    f(t) := \E \left[ P_v(t) \right]= \E \left( \langle \nu_v(t) , 1 \rangle + \langle \nuc(t) , 1 \rangle \right).
\end{equation*}
By Proposition \ref{MartingalesProp} we have
\begin{align*}
    f(t) &= f(0) + \int_0^t \E \left[ \int_{\Vp} b(x,z)  \nu_v(s)(dx,dz) - \int_{\Vp} \left( d(z) + c \langle \nu_v^{x}, 1\rangle \right)  \nu_v(s)(dx,dz) \right] ds, \nonumber \\
\end{align*}
and as a consequence we obtain the differentiability of $f$. Moreover, we also have:
\begin{equation*}
    f\myprime(t) \leq (\bar{b}-\hat{d}) f(t) - c \cdot \E \left[ \int_{\Vp}   \langle \nu_v^{x}(t), 1\rangle  \nu_v(t)(dx,dz) \right].
\end{equation*}
where $\bar{b}$ is given in Assumption \ref{AssPherates}, and $\hat{d}:= \inf_{z \in \X} d(z)$.\\

By Jensen's inequality we have:
\begin{align*}
    f(t)^2 &\leq \E \left[N_v(t)^2 + 2 N_c(t) N_v(t) + N_c(t)^2 \right] \nonumber \\
    &\leq \E \left[N_v(t)^2 + 2 V_0 \right( N_v(t) + N_c(t) \left) \right] \nonumber
\end{align*}
which can be rewritten as:
\begin{equation*}
    \E \left[N_v(t)^2 \right] \geq  f(t)^2 - 2 V_0 f(t), 
\end{equation*}
This, together with Lemma \ref{inequalarren}, implies
\begin{equation*}
    f\myprime(t) \leq \left( (\bar{b}-\hat{d}) + \frac{2 c V_0}{\lvert E\rvert} \right) f(t) - \frac{c}{\lvert E\rvert} f(t)^2.
\end{equation*}
From the observation that the function:
\begin{equation*}
    y(x) = \left(  (\bar{b}-\hat{d}) + \frac{2 c V_0}{\lvert E\rvert} \right) x - \frac{c}{\lvert E\rvert} x^2
\end{equation*}
is negative for any $x$ such that
\begin{equation*}
    x \geq x_0 := \lvert E\rvert \frac{ (\bar{b}-\hat{d})}{c}+ 2V_0, 
\end{equation*}
we deduce that
\begin{equation*}
    f(t) \leq f(0) \vee x_0
\end{equation*}
for all $t \geq 0$. This implies \eqref{step1}.\\

Now we claim that:
\begin{equation*}
    \lim_{t \to \infty} \,  P_v(t)   \in \lbrace 0, \infty \rbrace.
\end{equation*}
Using the fact that $P_v(t)$ is $\N$-valued, it is enough to check that for any $M \in \N$ we have
\begin{equation*}
    \P \left[ \liminf_{t \to \infty} \,  P_v(t) = M  \right] = 0 .
\end{equation*}
Assume that it is not the case, and that $\liminf_{t \to \infty} \,  P_v(t) = M $. By definition of $\liminf$, this implies that $P_v(t)$ reaches the value $M$ infinitely often, but the value $M-1$ only a finite number of times. However, this is almost surely impossible since every time that  the process $P_v$ is at state $M$, the probability of going to state $(M-1)$ is bounded from below by:
\begin{equation*}
    \frac{\hat{d} M}{\bar{b}M + \bar{d}M + c M^2 \bar{\gamma}V_0} > 0,
\end{equation*}
where we recall that $\hat{d}>0$. We only have to use the fact that $\lbrace 0 \rbrace$ is an absorbing state to deduce  that the limit exists and 
\begin{equation*}
    \lim_{t \to \infty } P_v(t) \in \lbrace 0, \infty \rbrace.
\end{equation*}
To conclude that a.s.
\begin{equation*}
    \lim_{t \to \infty} \,  P_v(t)  = 0,
\end{equation*}
 it is enough to show that
\begin{equation*}
    \E \left[ \lim_{t \to \infty}  P_v(t) \right] < \infty.
\end{equation*}
However, this is a consequence of Fatou's lemma, expression \eqref{step1}, and the following reasoning:
\begin{equation*}
      \E \left[ \lim_{t \to \infty}  P_v(t) \right]  =     \E \left[ \liminf_{t \to \infty}  P_v(t) \right]  \leq      \liminf_{t \to \infty} \, \E \left[  P_v(t) \right]  \leq      \sup_{t \geq 0} \, \E \left[  P_v(t) \right]< \infty.
\end{equation*}
\end{proof}

\section{IDE formulation}\label{Mainresultssect}
Let us introduce some additional notation needed to introduce the IDE formulations of Theorem \ref{Theoremequalintro} and Theorem \ref{Theoremlessintro}. 

\subsection*{Rescaled processes}
Let us denote by $\Lambda_K(t)$ the measure-valued process of the form:
\begin{equation}\label{Lambdakdef}
  \Lambda_K(t) = (\nuv^{(K)}(t), \nuu^{(K)}(t), \nuc^{(K)}(t) ),
\end{equation}
where $\nuv^{(K)}(t), \nuu^{(K)}(t)$, and $\nuc^{(K)}(t)$  are given as in \eqref{lambdak2}. We consider $\Lambda_K(t)$ as a process taking values in the product space of measures $\M_F$ given by:
 \begin{equation*}
     \M_F := \M_F(\Vp) \times \M_F(\D) \times \M_F(\Vc)
 \end{equation*}
 
Under Assumption \ref{Assbetakreg2}, the process  $\Lambda_K(t)$ has an infinitesimal  generator, that we denote by $\L^{(K)}$, with jump part $ \L_1^{(K)}$ given by:
\begin{align*}
 \L_1^{(K)} F_{\phivec}(\nuvec) &= K (1-\mu) \int_{\Vp} b(x,z) \left[F_{\phivec}( (\nuv + \tfrac{1}{K}\delta_{x,z}, \nuu, \nuc)) -F_{\phivec}(\nuvec) \right] \nuv(dx,dz) \nonumber \\
  &+ K\, \mu \int_{\Vp} b(x,z) \int_{\X} m(z,e) \, \left[F_{\phivec}( (\nuv + \tfrac{1}{K} \delta_{x,e}, \nuu, \nuc)) -F_{\phivec}(\nuvec) \right] \nuv(dx,dz) \, de \nonumber \\
  &+ K \int_{\Vp} (d(z)+ c \langle \nu_v^{x},1 \rangle ) \, \left[F_{\phivec}( (\nuv - \tfrac{1}{K}\delta_{x,z}, \nuu, \nuc)) -F_{\phivec}(\nuvec) \right] \nuv(dx,dz) \nonumber \\
  &+ K \int_{\Vp }\int_{D} \beta(t,y,x,N_x(t),z) \, \left[F_{\phivec}( (\nuv - \tfrac{1}{K}\delta_{x,z}, \nuu-\tfrac{1}{K^{\lambda}}\delta_y, \nuc+\tfrac{1}{K^{\lambda}}\delta_{y,z})) -F_{\phivec}(\nuvec) \right] \nuv(dx,dz) \, \nuu(dy) \nonumber \\
  &+ K \int_{\Vc} \int_{E} \eta(t,w,x,e) \, \left[F_{\phivec}( (\nuv + \tfrac{1}{K}\delta_{x,e}, \nuu+\tfrac{1}{K^\lambda}\delta_w, \nuc-\tfrac{1}{K^\lambda}\delta_{w,e})) -F_{\phivec}(\nuvec) \right] \nuc(dw,de) \, dx \nonumber \\
  &+ K \int_{\Vc} \gamma(e) \, \left[F_{\phivec}( (\nuv, \nuu+\tfrac{1}{K^\lambda}\delta_w, \nuc-\tfrac{1}{K^\lambda}\delta_{w,e})) -F_{\phivec}(\nuvec) \right] \nuc(dw,de), 
\end{align*}
and a diffusive part:
\begin{align*}
\L_2^{(K)} F_\phivec(\nuvec)  &:=K^{1-\lambda}\left(\int_{\D} \L_y^u \phiu(y) \nuu(dy)+\int_{\Vc} \Lc \phic(w,e) \nuc(dw,de) \right) \,  F_\phivec\myprime(\nuvec) \nonumber \\
&+K^{1-2\lambda}\left( \int_{\D} \frac{(\sigma^u(y))^2}{2} |\nabla_y \phiu(y) |^2  \nuu(dy) +\int_{\Vc} \frac{(\sigma^c(w))^2}{2} |\nabla_w \phic(w,e) |^2  \nuc(dw,de) \right)\, F_\phivec\mydprime(\nuvec).
\end{align*}

\subsection*{Population at time zero}
At time zero we consider a sequence of initial measures $\left\{ \Lambda^K(0) \right\}_{\{ k \geq 1 \}} \in \M_F$ of the form:
\begin{equation*}
\Lambda^{(K)}(0) = ( \nuv^{(K)}(0), \nuu^{(K)}(0), \nuc^{(K)}(0)  ).
\end{equation*}
We assume that this measure satisfies an estimate like the one needed in Proposition \ref{Propcontrolp}, and that a law of large numbers is satisfied by the sequence $\left\{ \Lambda^K(0) \right\}_{\{ k \geq 1 \}}$. More precisely we make the following assumption:
\begin{assumption}\label{HydroAssump}
The sequence of measures $\left\{  \Lambda^K(0) \right\}_{\{ K \geq 1 \}}$ converges in law and for the weak topology of $\M_F$ to a deterministic finite measure $\xi_0 = (\xi_v(0),\xi_S(0),\xi_I(0)) \in \M_F.$ Moreover we have the following estimate:
\begin{equation*}
   \sup_{K \in \N} \left(  \E \left[ \langle  \nu_\alpha^K(0), 1 \rangle^3 \right] \right) < \infty,
\end{equation*}
for all $\alpha \in \lbrace v, u, c \rbrace$.
\end{assumption}
The IDE formulation, given by Theorem \ref{IDElambdaequal} below, can be thought of as a law of large numbers, for the limit of the sequence of measures $\left\{\Lambda_K(t): t \geq 0 \right\}_{K \geq 1}$, seen as taking values in the path space $\mathbb{D}([0,T], \M_F)$, for all $T>0$. Obtaining such a description requires to work with martingales associated  to a Markov process and to control their quadratic variation. We will split the analysis of the martingales the two cases already described in Section \ref{IntroSect}.

\begin{remark}
It is possible to find combinations of rescaling of parameters, different from those of Assumption \ref{Assbetakreg2}, such that the speeding of diffusion by vectors becomes unnecessary to derive a limiting description for the case $\lambda <1$. However, for such a combination, the limiting evolution equation for the population of viruses decouples from that of the vectors (i.e. viruses do not see the effect of vectors). Such a description becomes biologically irrelevant in the context of vector-borne dynamics.
\end{remark}

\subsection{IDE formulation: first case}

 We have the following theorem:
\begin{theorem}\label{IDElambdaequal}
Let $\lambda=1$. Consider the sequence of measure-valued processes $\left\{\Lambda_K \right\}_{K \geq 1}$ given by \eqref{Lambdakdef}-\eqref{lambdak2}. Suppose Assumptions  \ref{Assbetakreg2}-\ref{HydroAssump} are satisfied. Then, for all $T>0$, the sequence of processes $\left\{\Lambda_K \right\}_{K \geq 1}$ converges in law in $\mathbb{D}([0,T], \M_F)$ to a deterministic continuous function $\xi(t) = \left( \xiv(t), \xiu(t), \xic(t) \right)$ belonging to the path space $\mathbb{C}([0,T], \M_F)$, and solving the following integro-differential equations:
\begin{align}\label{IDElambdaequaleq1}
\langle \xiv(t) , \phiv\rangle &=  \langle \xiv(0) , \phiv\rangle +\int_0^t \int_{\Vp} b(x,z) \phiv(x,z) \, \xiv(s)(dx,dz) \,  ds \nonumber \\
&- \mu \int_0^t \int_{\Vp} b(x,z)\left[ \phiv(x,z) -\int_{ \X} m(z,e) \phiv(x,e) \, de  \right] \xiv(s)(dx,dz) \,  ds \nonumber \\
&-  \int_0^t  \int_{\Vp}  \left[ d(z) + c \langle \xiv(s)^x, 1\rangle \right] \phiv(x,z) \xiv(s)(dx,dz) ds \nonumber \\
&- \int_0^t  \int_{\Vp \times \D} \beta(s,y,x,\langle \xiv(s)^x, 1\rangle,z) \phiv(x,z) \, \xiv(s)(dx,dz) \,  \xiu(s)(dy) \, ds \nonumber \\
&+ \int_0^t   \int_{\Vc\times E} \eta(s,y,x,z) \phiv(x,z)  \,  \xic(s)(dy,dz) \, dx \,   ds,
\end{align}
\begin{align}\label{IDElambdaequaleq2}
\langle \xiu(t) , \phiu\rangle &=  \langle \xiu(0) , \phiu\rangle +\int_0^t \int_{\D} \mathcal{L}^u \phiu(y) \, \xiu(s)(dy) \, ds \nonumber \\
&-  \int_0^t  \int_{\Vp \times \D} \beta(s,y,x,\langle \xiv(s)^x, 1\rangle,z) \phiu(y) \xiv(s)(dx,dz) \,  \xiu(s)(dy) \, ds \nonumber \\
&+\int_0^t   \int_{\Vc \times E} \eta(s,y,x,z) \phiu(y) \xic(s)(dy,dz) \,  dx \,   ds + \int_0^t \int_{\Vc} \gamma(z) \phiu(y) \, \xic(s)(dy,dz) ds, 
\end{align}
\begin{align}\label{IDElambdaequaleq3}
\langle \xic(t) , \phic\rangle&= \langle \xic(0) , \phic\rangle +\int_0^t \int_{\Vc} \mathcal{L}^c \phic(y,z) \, \xic(s)(dy,dz) \, ds \nonumber \\
&+ \int_0^t  \int_{\Vp \times \D} \beta(s,y,x,\langle \xiv(s)^x, 1\rangle,z) \phic(y,z) \xiv(s)(dx,dz) \,  \xiu(s)(dy) \, ds \nonumber \\
&-  \int_0^t   \int_{\Vp \times E} \eta(s,y,x,z) \phic(y,z) \xic(s)(dy,dz) \,  dx \,   ds -\int_0^t \int_{\Vc} \gamma(z) \phic(y,z) \, \xic(s)(dy,dz) \,  ds,
\end{align}
for all $\phiv $ bounded and measurable, $\phiu \in \C_b^2(\D) \cap \D(\L^u)$, and $\phic \in \C_b^2(\I) \cap D(\L^c)$.
\end{theorem}

We postpone the proof of Theorem \ref{IDElambdaequal} to Section \ref{genproofs} for the generalities and Section \ref{firstproof} for the particulars.

\begin{remark}\label{RemContrlReg1}
 For $\alpha \in \lbrace v, u \rbrace$, using the functions $\phi_\alpha=1$ and neglecting the negative terms, we can obtain the following bound:
\begin{equation*}
\max_{\alpha} \left( \langle \xi_\alpha(t), 1\rangle \right) \leq \max_{\alpha} \left( \langle \xi_\alpha(0), 1\rangle \right) + C \int_0^t \max_{\alpha} \left( \langle \xi_\alpha(s), 1\rangle \right) \, ds
\end{equation*}
by Gronwall's lemma we then conclude:
\begin{equation*}
    \sup_{t \in [0,T]} \max_{\alpha} \left( \langle \xi_\alpha(t), 1\rangle \right) \leq \max_{\alpha} \left( \langle \xi_\alpha(0), 1\rangle\right) e^{CT}
\end{equation*}
which by the conservation of mass in vectors, the fact that the sup is bounded can also be extended to the case $\alpha=c$.
\end{remark}

Theorem \ref{Theoremequalintro} is the strong form of Theorem \ref{IDElambdaequal}. In order to verify that we can indeed obtain such a strong formulation, we need to verify that if we assume that at time zero the triplet $\left( \xiv(t), \xiu(t), \xic(t) \right)$ has densities with respect to their respective Haar measures, indeed we have the propagation of this property for later times. 

\subsection{Propagation of absolute continuity}

\begin{theorem}\label{PropAbsol}
Let us assume that for  $ \alpha \in \lbrace u, c \rbrace$ the diffusive generator $\L^\alpha$ is reversible. Let us also assume that for $ \alpha \in \lbrace v, u, c \rbrace$, at time zero each of the limiting measures $\xi_\alpha(0)$ admits a density of the following form: 
\begin{align*}
\xi_v(0)(dx,dz) = g_v(x,z) \, \zeta_E(dx) \, dz, \quad \xiu(0)(dy,dz) = g_u(y) \, dy, \text{ and } \xic(0)(dy,dz) = g_c(y,z) \, dy \, dz,
\end{align*}
where $\zeta_E$ denotes the counting measure on $E$. Consider the associated solution  $\xi(t) = \left( \xi_v(t), \xiu(t), \xic(t) \right)$ of \eqref{IDElambdaequaleq1}-\eqref{IDElambdaequaleq3} with initial condition $ \left( \xi_v(0), \xiu(0), \xic(0) \right)$. Then, under the assumptions of Theorem \ref{IDElambdaequal}, for each time $t >0$, and each $\alpha \in \lbrace v, u, c \rbrace$, there exists functions $g_\alpha(t, \cdot): \D_\alpha \to \R$ such that:
\begin{align*}
\xi_v(t)(dx,dz) = g_v(t,x,z) \, dz, \quad 
\xiu(t)(dy,dz) = g_u(t,y) \, dy, \text{ and } 
\xic(t)(dy,dz) = g_c(t,y,z) \, dy \, dz. \nn 
\end{align*}
In other words, we have propagation of absolute continuity.
\end{theorem}

\begin{proof}
Since the space $E$ is finite, and $\zeta_E$ is the counting measure on $E$, any $\zeta_E$-null subset of $E$ is empty. Hence we only analyse null sets of $\X$ and $\D$. Let us consider a set $A \subset \X$ of Lebesgue measure zero. We want to show that both measures $\xi^\X_v(t) := \int_E \xi_v(t)(dx,\cdot)$ and $\xi^\X_c(t):= \int_\D \xic(dy,\cdot)$ assign zero mass to $A$. We want to take advantage of the fact that both measures are non-negative to only deal with the measure:
\begin{equation*}
    \nu^\X(t) = \xi^\X_v(t) + \xi^\X_c(t)
\end{equation*}
By Theorem \eqref{IDElambdaequal}, and the absolute continuity at time zero, we have
\begin{align}\label{abscontnuX0}
 \nu^\X(t)(A)  &=   \int_0^t \int_{\Vp} b(x,z) \indic{A}(z) \, \xi_v(s)(dx,dz) \,  ds -  \int_0^t  \int_{\Vp}  \left[ d(z) + c \langle \xi_v(s)^x, 1\rangle \right] \, \indic{A}(z) \xi_v(s)(dx,dz) ds \nonumber \\
& -\int_0^t \int_{\Vc} \gamma(z) \indic{A}(z) \, \xic(s)(dy,dz) \,  ds
\end{align}
Let the function $M(t)$ be given by
\begin{equation*}
    M(t) :=  \nu^\X(t)(A)
\end{equation*}
Notice that $M(0)=0$, and by \eqref{abscontnuX0}  $M(t)$  is differentiable.\\

Moreover, by Assumption \ref{AssPherates} and the non-negativity of $\xi^\X_c(t)$, we have:
\begin{align*}
 \frac{d}{dt}M(t)  &\leq   \bar{b} M(t) -  \int_{\Vp}  \left[ d(z) + c \langle \xi_v(t)^x, 1\rangle \right] \, \indic{A}(x,z) \xi_v(t)(dx,dz)  \nonumber \\
&-\int_0^t \int_{\Vc} \gamma(z) \indic{A}(z) \, \xic(s)(dy,dz) \,  ds \nn 
\end{align*}
which implies $ \frac{d}{dt}M(t)\leq   \bar{b} M(t)$,
and as a consequence $ M(t) = 0$.\\

We now show absolute continuity of $\xiu(t)$  with respect to the Lebesgue measure on $\D$. For simplicity of exposition let us assume that $\L^u = \L^c$ (denoted by $ \L$), and consider a measurable set $D \subset \D$ of null Lebesgue measure. Here we also take advantage of the positivity of both measures $\xiu(t)$ and $\xi_c^\D(t):= \int_\X \xic(t)(\cdot, dz)$, to only deal with the measure:
\begin{equation*}
    \nu^\D(t) = \xiu(t) +\xi_c^\D(t)
\end{equation*}

We would like to use to plug $\indic{D}$ and obtain an expression for $\nu^\D(t)(D)$, but the indicator function $\indic{D}$ is not an element of $D(\L)$. However we can use the semi-group trick that we use in Section \ref{UNIQUESEC} to define the following function:
\begin{equation*}
  \phi(s, y) = P_\L (t-s) \indic{D}(y) \quad \forall y \in \D,
\end{equation*}
for a fixed $t \in [0,T]$. By construction we have that $\phi(s, y)$ is a solution of the boundary problem:
\begin{align*}
    \partial_s \phi(s, y) &- \L \phi(s, y) = 0 \quad \text{ on } [0,T] \times \D  \nonumber \\
    \partial_n \phi(s, y) &= 0 \quad \text{ on } [0,T] \times \partial \D,  \nonumber \\
\lim_{s \to t}    \phi(s, y) &=  \indic{D}( y) \quad \text{ on } \D. 
\end{align*}
Notice that we can also re-write equations \eqref{IDElambdaequaleq1}-\eqref{IDElambdaequaleq3} in their weak time-space formulation. In particular for the last two equations we obtain:
\begin{align}\label{WEAKST2}
\langle \xiu(t) , \phiu(t,\cdot)\rangle &=  \langle \xiu(0) , \phiu(0,\cdot) \rangle +\int_0^t \int_{\mathcal{D}} \left(\mathcal{L}^u \phiu(s,y) +\partial_s  \phiu(s,y) \right)\, \xiu(s)(dy) \, ds \nonumber \\
&-  \int_0^t  \int_{\Vp \times \D} \beta(s,y,x,\langle \xi_v^x(s), 1\rangle,z) \phiu(s,y) \xiv(s)(dx,dz) \,  \xiu(s)(dy) \, ds \nonumber \\
&+\int_0^t   \int_{\Vc \times E} \eta(s,y,x,z) \phiu(s,y) \xic(s)(dy,dz) \,  dx \,   ds + \int_0^t \int_{\Vc} \gamma(z) \phiu(s,y) \, \xic(s)(dy,dz) ds, 
\end{align}
and 
\begin{align}\label{WEAKST3}
\langle \xic(t) , \phic(t,\cdot)\rangle&= \langle \xic(0) , \phic(0,\cdot)\rangle +\int_0^t \int_{\Vc} \left(\mathcal{L}^c \phic(s,y,z) + \partial_s \phic(s,y,z)  \right) \, \xic(s)(dy,dz) \, ds \nonumber \\
&+ \int_0^t  \int_{\Vp \times \D} \beta(s,y,x,\langle \xi_v^x(s), 1\rangle,z) \phic(s,y,z)  \xiv(s)(dx,dz) \,  \xiu(s)(dy) \, ds \nonumber \\
&-  \int_0^t   \int_{\Vc \times E} \eta(s,y,x,z) \phic(s,y,z)  \xic(s)(dy,dz) \,  dx \,   ds -\int_0^t \int_{\Vc} \gamma(z) \phic(s,y,z)  \, \xic(s)(dy,dz) \,  ds. 
\end{align}
Summing \eqref{WEAKST2} and \eqref{WEAKST3}, and unsing that $\L^u=\L^c=\L$, gives:
\begin{align}\label{nuequation}
\langle \nu^\D(t)  , \indic{D} \rangle &=  \langle \nu_0 , \phi(0,\cdot) \rangle = \int_{\D} g_u(y) P_\L (t) \indic{D}(y) \, dy + \int_{\D} \left( \int_{\X} g_c(y,z) \, dz \right) P_\L (t) \indic{D}(y) \, dy = 0,
\end{align}
where in the last equality we used the reversibility of $\L$.

Finally we check that the measure $\xic(t)$ is absolutely continuous with respect to the product of the Lebesgue measures on $\D$ and $\X$. Let us denote by $D_c$ a Lebesgue null-subset of $\D \times \X$. Similar than in the previous case, we define the following function:
\begin{equation*}
  \phi(s, y,z) = P_{\L_c} (t-s) \indic{D_c}(y,z) \quad \forall (y,z) \in \D \times \X,
\end{equation*}
for a fixed $t \in [0,T]$ which solves its corresponding boundary problem. We then conclude by noticing that in this case, by the non-negativity of $\xic$ we have:
\begin{align}
    \xic(t)(D_c) &= \int_{\D \times \X} \, g_c(y,z) \,  P_{\L_c} (t) \indic{D_c}(y,z) \, dy \, dz \nonumber \\
&-  \int_0^t   \int_{\Vc \times E} \eta(s,y,x,z) \phic(s,y,z)  \xic(s)(dy,dz) \,  dx \,   ds -\int_0^t \int_{\Vc} \gamma(z) \phic(s,y,z)  \, \xic(s)(dy,dz) \,  ds \leq 0.
\end{align}
\end{proof}

\subsection{Proof of Theorem \ref{PersCond}}\label{ProofLocPer}

Let us start by giving the linearization of the system given in Theorem \ref{Theoremequalintro} around a generic equilibrium point $(g_v^*,g_u^*,g_c^*)$:
\begin{align*}
\frac{\partial}{\partial t} h_v(t,x,z) &=  (1-\mu) b(x,z) \, h_v(t,x,z)+ \mu \int_\X b(x,z\myprime) m(z^\prime, z) \,  h_v(t,x,z\myprime) \, dz^\prime   -  d(z) h_v(t,x,z)  \nonumber \\
&- c \left( \int_\X g_v^*(x,z^\prime) dz\myprime \right) h_v(t,x,z) - c \left( \int_\X h_v(t,x,z^\prime) dz\myprime \right) g_v^*(x,z)      \nonumber \\
&- \left(\int_\D \beta(y,x,z) g_u^*(y) \, dy \right) \, h_v(t,x,z) - \left(\int_\D \beta(y,x,z) h_u(t,y) \, dy \right) \, g_v^*(x,z) \nonumber \\
&+ \int_{\D} \eta(y,x,z) \,  h_c(t,y,z) \, dy
\end{align*}
\begin{align*}
\frac{\partial}{\partial t} h_u(t,y)  &=  \Delta h_u(t,y) - \left(\int_{ E \times \X} \beta(y,x,z) \,  h_v(t,x,z) \, \zeta_E( dx) \, dz \right)   \, g_u^*(y)   \nonumber \\
& - \left(\int_{ E \times \X} \beta(y,x,z) \,  g_v^*(x,z) \, \zeta_E( dx) \, dz \right)   \, h_u(t,y) +\int_{E \times \X} \eta(y,x,z) \,  h_c(t,y,z) \, \zeta_E( dx) \, dz 
\end{align*}
\begin{align*}
\frac{\partial}{\partial t} h_c(t,y,z)&= \Delta_y h_c(t,y,z)+ \left(\int_{ E } \beta(y,x,z) \,  h_v(t,x,z) \, \zeta_E( dx)  \right)   \, g_u^*(y)  \nonumber \\
&+ \left(\int_{ E } \beta(y,x,z) \,  g_v^*(x,z) \, \zeta_E( dx)  \right)   \, h_u(t,y)-\left(\int_{E} \eta(y,x,z) \, \zeta_E( dx) \right)\,  h_c(t,y,z) 
\end{align*}
with the same boundary conditions as before.\\

We will further specialize the previous linearization to the  equilibrium points:
\begin{align*}
  g_v^*(x,z)=0, \quad   g_u^*(y) = 1, \text{ and } g_c^*(y,z)=0
\end{align*}
for all $x \in E, y \in \D$, and $z \in \X$. To conclude local persistence it is enough to show that the perturbation $h_v$ is monotonically increasing with time.

For this particular set of equilibria the linearization takes the form:
\begin{align}\label{IDE1strlinequi}
\frac{\partial}{\partial t} h_v(t,x,z) &= \left( (1-\mu) b(x,z)  -  d(z) -\int_\D \beta(y,x,z)  \, dy\right)\, h_v(t,x,z)   \nonumber \\
&+ \mu \int_\X b(x,z\myprime) m(z^\prime, z) \,  h_v(t,x,z\myprime) \, dz^\prime  +\left( \int_{\D} \eta(y,x,z) \,  h_c(t,y,z) \, dy \right)
\end{align}
\begin{align}\label{IDE2strlinequi}
\frac{\partial}{\partial t} h_u(t,y)  &=  \Delta h_u(t,y) - \left(\int_{ E \times \X} \beta(y,x,z) \,  h_v(t,x,z) \, \zeta_E( dx) \, dz \right)   \nonumber \\
&+\int_{E \times \X} \eta(y,x,z) \,  h_c(t,y,z) \, \zeta_E( dx) \, dz  
\end{align}
\begin{align}\label{IDE3strlinequi}
\frac{\partial}{\partial t} h_c(t,y,z)&= \Delta_y h_c(t,y,z)+ \left(\int_{ E } \beta(y,x,z) \,  h_v(t,x,z) \, \zeta_E( dx)  \right)    \nonumber \\
&-\left(\int_{E} \eta(y,x,z) \, \zeta_E(dx) \right)\,  h_c(t,y,z) 
\end{align}
with the same boundary conditions as before.\\

By Theorem 2.2  in \cite{sato1965multi} we have that
\begin{align}
 h_c(t,y,z) = \int_{\D}p(t,y,y\myprime;z) h_c(0,y,z) \, dy\myprime + \int_0^t \int_{\D}p(s,y,y\myprime;z)  \left(\int_{ E } \beta(y\myprime,x,z) \,  h_v(t-s,x,z) \, \zeta_E( dx) \right) \, dy\myprime \, ds
\end{align}
where for every $z \in \X$, $p(t,\cdot,\cdot,z)$ is the fundamental solution to:
\begin{align*}
\frac{\partial}{\partial t} h_c(t,y,z)&= \Delta_y h_c(t,y,z)-\left(\int_{E} \eta(y,x,z) \, \zeta_E(dx) \right)\,  h_c(t,y,z)
\end{align*}
with same boundary conditions as before.\\

Assuming $h_c(0,y,z) =0$, reduces \eqref{IDE1strlinequi} to:
\begin{align}\label{eqnequihv}
\frac{\partial}{\partial t} h_v(t,x,z) &= \left( (1-\mu) b(x,z)  -  d(z) -\int_\D \beta(y,x,z)  \, dy\right)\, h_v(t,x,z) + \mu \int_\X b(x,z\myprime) m(z^\prime, z) \,  h_v(t,x,z\myprime) \, dz^\prime    \nonumber \\
&+ \int_{\D} \eta(y,x,z) \,  \left(\int_{ E } \int_0^t  \left(  \int_{\D}p(s,y,y\myprime;z)  \beta(y\myprime,x,z) \, dy\myprime \, \right) \,  h_v(t-s,x\myprime,z) \, ds  \, \zeta_E(dx\myprime) \right) \, dy.
\end{align}
The positivity of the RHS of \eqref{eqnequihv}, which is a consequence of \eqref{condongvcero} and \eqref{condonR}, concludes the proof. 


\subsection{IDE formulation: second case}
Suppose Assumptions  \ref{Assbetakreg2} and \ref{HydroAssump} are satisfied. Then we have the following:
\begin{theorem}\label{IDElambdaless}
Let $\lambda \in (0,1)$. Consider the sequence of processes $\left\{\nuv^K(t): t \geq 0 \right\}_{K \geq 1}$ given by \eqref{Lambdakdef}. Then for all $T>0$, the sequence $\left\{\nuv^K(t): t \geq 0 \right\}_{K \geq 1}$ converges in law in $\mathbb{D}([0,T], \M_F(\Vp))$ to a deterministic continuous function $\xiv$ belonging to the space $\mathbb{C}([0,T], \M_F(\Vp))$, and solving the IDE:
\begin{align}\label{IDEvirusequal}
\langle \xiv(t) , \phiv\rangle &=  \langle \xiv(0) , \phiv\rangle +\int_0^t \int_{\Vp} b(x,z) \phiv(x,z) \, \xiv(s)(dx,dz) \,  ds \nonumber \\
&- \mu \int_0^t \int_{\Vp} b(x,z)\left[ \phiv(x,z) -\int_{ \X} m(z,e) \phiv(x,e) \, de  \right] \xiv(s)(dx,dz) \,  ds \nonumber \\
&-  \int_0^t  \int_{\Vp}  \left[ d(z) + c \langle \xiv(s)^x, 1\rangle \right] \phiv(x,z) \xiv(s)(dx,dz) ds \nonumber \\
&- \int_0^t  \int_{\Vp} \left( \int_{\M_F(\D)} \int_{ \D} \beta(s,y,x,\langle \xiv(s)^x, 1\rangle,z) \, \Pi_{\xiv(s)}^{B}(\xiu (dy) \times \M_F(\Vc)) \right) \, \phiv(x,z) \, \xiv(s)(dx,dz)  \, ds \nonumber \\
&+ \int_0^t   \int_{ E} \left( \int_{\M_F(\Vc)} \int_{\Vc}\eta(s,y,x,z) \, \Pi_{\xiv(s)}^{B}(\M_F(\D) \times \xic(dy,dz)) \right)\phiv(x,z)   \, dx \,   ds, 
\end{align}
for all $\phiv$ bounded and measurable, and where for each $\nuv \in \M_F(\Vp)$ the measure $\Pi_{\nu_v}^B$ is the unique stationary measure of the generator $B_{\nuv}:D(B_{\nuv}) \to \C(\D \times \I)$ given by:
\begin{align}\label{DefBnu}
B_{\nu_v} F_{\phiu,\phic}(\nuu,\nuc) &= \partial_x F_{\phiu,\phic}(\nuu,\nuc)  \int_{\D} \L^u \phiu(y) \, \nuu(dy)  +\partial_y F_{\phiu,\phic}(\nuu,\nuc) \int_{\Vc} \L^c \phic(w,e) \, \nuc(dw,de)  \nonumber \\
&- \partial_x F_{\phiu,\phic}(\nuu,\nuc)  \int_{\Vp \times \D} \beta(t,y,x,N_x(t),z) \phiu(y) \nuv(dx,dz) \,  \nuu(dy) \nn \\
&+ \partial_y F_{\phiu,\phic}(\nuu,\nuc)  \int_{\Vp \times \D} \beta(t,y,x,N_x(t),z) \phic(y,z)  \nuv(dx,dz) \,  \nuu(dy) \nonumber \\
&+\partial_x F_{\phiu,\phic}(\nuu,\nuc) \int_{\Vc \times E} \eta(t,y,x,z)  \phiu(y) \nuc(dy,dz) \,  dx  \nonumber \\
&-\partial_y F_{\phiu,\phic}(\nuu,\nuc) \int_{\Vc \times E} \eta(t,y,x,z) \phic(y,z)  \nuc(dy,dz) \,  dx  \nonumber \\
&+ \partial_x F_{\phiu,\phic}(\nuu,\nuc) \int_{\Vc} \gamma(z)  \phiu(y)  \nuc(dy,dz) -\partial_y F_{\phiu,\phic}(\nuu,\nuc) \int_{\Vc} \gamma(z) \phic(y,z) \nuc(dy,dz) 
\end{align}
for all $\phiu \in \C_b^2(\D) \cap \D(\L^u)$, and $\phic \in \C_b^2(\I) \cap D(\L^c)$.
\end{theorem}


\subsection{Absolute continuity and  stationarity of the measure $\Pi_{\nu_v}^B$}

Theorem \eqref{IDElambdaless} seems rather abstract due to the apparently lack of information about the measure $\Pi_{\nu_v}^B$. However, under additional conditions we can find an explicit expression for the measure $\Pi_{\nu_v}^B$. Let us then make the following assumption:

\begin{assumption}\label{betahatassumption}
Assume that for all $\nuv \in \M_F(\Vp)$, and all $t \geq 0$, the function given by:
\begin{equation}\label{hatbeta}
    \bar{\beta}(t,y) = \int_{\Vp} \beta(s,y,x,\langle \nu_v^x, 1\rangle,z) \nu_v(dx,dz).
\end{equation}
is in $\C^2(\D)$. Moreover, assume the existence of a constant $\beta_0>0$ such that
\begin{equation}\label{betazerobound}
    \hat{\beta}(t,y) \geq \beta_0 > 0,
\end{equation}
uniformly in $y$.
Finally, assume that the death rate of viruses on vectors is zero.
\end{assumption}

First notice that the measure  $\Pi_{\nu_v}^B$ satisfies
\begin{equation}\label{statPinu}
    \int_{\M_F(\D) \times \M_F(\Vc)}  B_{\nu_v} F_{\phiu,\phic}(\nuu,\nuc) \,  \Pi_{\nu_v}^B(d \nuu, d \nuc)= 0 
\end{equation}
for all cylindrical function $F_{\phiu,\phic}: \M_F(\D) \times \M_F(\Vc) \to \R$.\\

Equation \eqref{statPinu} and the precise form \eqref{DefBnu}  implies that the support of the measure $\Pi_{\nu_v}^B$ is inside the set of all measures $\nuu$ $\nuc$ such that: 
\begin{align}\label{EcuaDefSupp0}
&\int_{\D} \L^u \phiu(y) \, \nuu(dy)- \int_{\Vp \times \D} \beta(t,y,x,N_x(t),z) \phiu(y) \nuv(dx,dz) \,  \nuu(dy)   \nn \\
&+ \int_{\Vc \times E} \eta(t,y,x,z)  \phiu(y) \nuc(dy,dz) \,  dx + \int_{\Vc} \gamma(z)  \phiu(y)  \nuc(dy,dz)  = 0,
\end{align}
and
\begin{align}\label{EcuaDefSupp1}
 &\int_{\Vc} \L^c \phic(w,e) \, \nuc(dw,de) +\int_{\Vp \times \D} \beta(t,y,x,N_x(t),z) \phic(y,z)  \nuv(dx,dz) \,  \nuu(dy)  \nn \\
& - \int_{\Vc \times E} \eta(t,y,x,z) \phic(y,z)  \nuc(dy,dz) \, \zeta_E(dx) - \int_{\Vc} \gamma(z) \phic(y,z) \nuc(dy,dz)= 0,
\end{align}
for all $\phiu \in \C_b^2(\D) \cap \D(\L^u)$, and $\phic \in \C_b^2(\I) \cap D(\L^c)$.

Since the system \ref{EcuaDefSupp0}-\ref{EcuaDefSupp1} has a unique solution $(\xi_u,\xi_c) \in \M_F(\D) \times \M_F(\Vc)$, we have that
\begin{equation*}
\Supp \, \Pi_{\nu_v}^B \subseteq \lbrace (\xi_u,\xi_c)  \rbrace.
\end{equation*}
The fact that $\Pi_{\nu_v}^B$ is a probability measure implies that it is of the form $\Pi_{\nu_v}^B = \delta_{\xiu,\xic}$, where $\xiu \in \M_F(\D)$, and $\xic \in \M_F(\Vc)$ satisfy \ref{EcuaDefSupp0}-\ref{EcuaDefSupp1}.\\

We want to find conditions that guarantee that the measures $\xiu$ and $\bar{\xi}_c$ have a density with respect to the Lebesgue measure on $\D$. Where $\bar{\xi}_c$ is given by:
\begin{equation*}
    \bar{\xi}_c (dy) := \int_{\X} \xic(dy,dz).
\end{equation*}
Let $W \subset \D$ be a set of Lebesgue measure zero, we want to find conditions such that:
\begin{align}
    \xiu(W) &= \int_{W} 1 \cdot \xiu(dy) = 0 \\ \label{LebDensCondII}
    \bar{\xi}_c(W) &= \int_{W} 1 \cdot \bar{\xi}_c(dy) = 0.    
\end{align}
Let then $f_W^u \in D(\L^u)$ satisfy the following PDE:
\begin{equation}\label{PDERefProb}
\L^u f_W^u(y) - \bar{\beta}(t,y) f_W^u(y) = \indic{W}(y), \text{ with }  \nabla f_W^u \cdot \bar{n}(y) = 0, \quad \forall y \in \partial \D,
\end{equation}
and where $\hat{\beta}$ is given by \eqref{hatbeta}.

Notice that from \ref{EcuaDefSupp0}-\ref{EcuaDefSupp1} with $\phiu = f_W^u$, and $\phic= 0$, we obtain:
\begin{equation}\label{Lebesbeta}
\xiu(W) =\int_{W} 1 \cdot \xiu(dy)  =-   \int_{E \times \D \times \X} \eta(t,y,x,z)  f_W^u(y) \xic(dy,dz) \, \zeta_E( dx).
\end{equation}
Using Assumption \ref{betahatassumption} we guarantee, first the existence of such a solution $f_W^u$, and second, that the RHS of \eqref{Lebesbeta} vanishes. By It\^o's formula, and the product rule (see Section 5.7 of \cite{karatzas2012brownian} for a similar procedure), we can find the following probabilistic interpretation of the solution of \eqref{PDERefProb}:
\begin{equation*}
 f_W^u(y) = -\E \left[ \int_0^\infty e^{-\int_0^s \bar{\beta}(t,Y_r^y) dr} \indic{W}(Y_s^y)ds\right]   
\end{equation*}
where $Y_s^y$ is the It\^o diffusion with infinitesimal generator $\L^u$, and started at $y \in \D$. Then we have:
\begin{align*}
 \xiu(W)  &=-  \int_{E \times \D \times \X} \eta(t,y,x,z)  f_W^u(y) \xic(dy,dz) = \int_{E \times \D \times \X} \eta(t,y,x,z) \E \left[ \int_0^\infty e^{-\int_0^s \bar{\beta}(t,Y_r^y) dr} \indic{W}(Y_s^y)ds\right] \xic(dy,dz) \nonumber \\
    &\leq \int_{E \times \D \times \X} \eta(t,y,x,z)  \left( \int_0^\infty  \P\left( Y_s^y \in W\right) \, ds\right) \xic(dy,dz)= 0
\end{align*}
where we have used Fubini's theorem, and the continuity of the diffusion parameters to be able to use that the transition kernel $P_s(y, \cdot)$, of the process $Y_s$, is absolutely continuous with respect to Lebesgue for each $s>0$ (see for example \cite{chen2011symmetric} for the case of RBM).\\

Analogously, to obtain that \eqref{LebDensCondII} is zero, we can also use the equation
\begin{equation*}
\L^c f_W^c(y,z) - \int_{E} \eta(t,y,x,z) \, \zeta_E( dx) \, f_W^c(y,z) = \indic{W}(y).
\end{equation*}

\section{Proofs of Theorem \ref{IDElambdaequal} and Theorem \ref{IDElambdaless} }\label{Proofs}

For the proofs of Theorem \ref{IDElambdaequal}
and Theorem \ref{IDElambdaless} we use the standard compactness-uniqueness approach. The proof of both theorems uses the same type of techniques with the exception of the part dealing with the characterization of limit points. The characterization of limit points for Theorem \ref{IDElambdaless} is interesting in its own. It is an application of an averaging principle due to T. Kurtz and given originally in \cite{kurtz1992averaging}. 

\subsection{Generalities for both cases}\label{genproofs}

\subsubsection*{Uniqueness of the limits}\label{UNIQUESEC}
Here we will combine the arguments for uniqueness used in the proof of Theorem 5.3 from \cite{fournier_microscopic_2004}, and in the proof of Theorem 4.2 in \cite{champagnat_invasion_2007}. Let us first assume that $(\xiv(t),\xiu(t),\xic(t))_{t \geq 0}$, and $(\bar{\xi}_v(t),\bar{\xi}_u(t),\bar{\xi}_c(t))_{t \geq 0}$ are solutions of \eqref{IDElambdaequaleq1}-\eqref{IDElambdaequaleq3}. We want to show that for all $\alpha \in \lbrace v, u, c \rbrace$ we have:
\begin{equation*}
    \norm{\xi_\alpha - \bar{\xi}_\alpha}_{\alpha} = 0,
\end{equation*}
where for $\nu_\alpha^1$ and $\nu_\alpha^2 \in \M_F(D_\alpha)$ their variation norm is given by:
\begin{equation*}
    \norm{\nu_\alpha^1 - \nu_\alpha^2}_\alpha = \sup_{\substack{\phi_\alpha \in L^\infty(\V_\alpha) \\ \norm{\phi_\alpha}_\infty\leq1}} \lvert \langle \nu_\alpha^1 - \nu_\alpha^2 , \phi_\alpha \rangle_\alpha \rvert. 
\end{equation*}
where $\V_u = \D$.

Let us first deal with the case $\alpha= v$. Let $\phiv$ be such that $\norm{\phiv}_\infty\leq 1$, by \eqref{IDElambdaequaleq1}  we have:
\begin{align}\label{uniquennxiv}
\left\lvert \langle \xi_v(t) -\bar{\xi}_v(t) , \phi_v\rangle \right\rvert &\leq (1-\mu) \int_0^t \left\lvert \int_{\Vp} b(x,z) \phiv(x,z) \, \left[\xi_v(s)(dx,dz)-\bar{\xi}_v(s)(dx,dz)\right] \,  \right\rvert ds \nonumber \\
&\hspace{-1.5cm}+ \mu \int_0^t \left\lvert \int_{\Vp} b(x,z)\int_{ \X} m(z,z^\prime) \phiv(x,z^\prime) \, dz^\prime  \, \left[\xi_v(s)(dx,dz)-\bar{\xi}_v(s)(dx,dz) \right] \right\rvert \,  ds \nonumber \\
&\hspace{-3cm}+  \int_0^t  \left\lvert \int_{\Vp}  d(z)  \phi_v(x,z) \left[\xi_v(s)(dx,dz)-\bar{\xi}_v(s)(dx,dz)\right] \right\rvert ds + c \int_0^t  \left\lvert\int_{\Vp}  \langle \xiv^x(s), 1\rangle \,  \phi_v(x,z) \left[\xi_v(s)(dx,dz)-\bar{\xi}_v(s)(dx,dz)\right] \right\rvert ds \nonumber \\
&+ c \int_0^t  \left\lvert\int_{\Vp}  \left[ \langle \bar{\xi}_v^x(s), 1\rangle- \langle \xiv^x(s), 1\rangle \right]  \phi_v(x,z) \,\bar{\xi}_v(s)(dx,dz) \right\rvert ds \nonumber \\
&\hspace{-3cm}+ \int_0^t  \left\lvert \int_{\Vp \times \D} \beta(s,y,x,\langle \xi_v^x(s), 1\rangle,z) \,  \xiu(s)(dy) \, \phi_v(x,z) \, \left[\xi_v(s)(dx,dz)-\bar{\xi}_v(s)(dx,dz)\right]  \right\rvert \, ds \nonumber \\
&\hspace{-3cm}+ \int_0^t  \left\lvert \int_{\Vp \times \mathcal{D}} \left[ \beta(s,y,x,\langle \xi_v^x(s), 1\rangle,z)  -\beta(s,y,x,\langle \bar{\xi}_v^x(s), 1\rangle,z)  \right] \, \xiu(s)(dy) \, \bar{\xi}_u(s)(dy) \, \phiv(x,z) \, \bar{\xi}_v(s)(dx,dz)  \right\rvert \, ds \nonumber \\
&\hspace{-3cm}+ \int_0^t  \left\lvert \int_{\Vp \times \mathcal{D}} \beta(s,y,x,\langle \bar{\xi}_v^x(s), 1\rangle,z) \, \left[ \xiu(s)(dy) -  \bar{\xiu}(s)(dy) \right]  \, \phiv(x,z) \, \bar{\xi}_v(s)(dx,dz)  \right\rvert \, ds \nonumber \\
&+ \int_0^t   \int_{\Vc \times E} \eta(s,y,x,z) \phi_v(x,z) \, dx  \,  \left[\xic(s)(dy,dz)- \bar{\xi}_c(s)(dy,dz)\right]  \,   ds.
\end{align}
Notice that the linear terms in the RHS of \eqref{uniquennxiv} can be bounded from above using $\norm{\phi_v}_\infty\leq 1$, and Assumption \ref{AssPherates}. For example the first term can be bounded from above as
\begin{equation*}
\int_0^t \left\lvert \int_{\Vp} b(x,z) \phi_v(x,z) \, \left[\xi_v(s)(dx,dz)-\bar{\xi}_v(s)(dx,dz)\right] \,  \right\rvert ds \leq \bar{b} \, \int_0^t \sup_{\phi \leq 1}  \,  \lvert \langle \xi_v(s) - \bar{\xi}_v(s), \phi \rangle \rvert  \, ds.  
\end{equation*}
Similar bounds can be derived for the other linear terms. The main difficulty comes from the non-linear terms. Let us first deal with the terms coming from competition among viruses, i.e. the integrals  with the competition constant $c$ in front. First notice that we have
\begin{align*}
\left\lvert \langle \bar{\xi}_v^x(s), 1\rangle- \langle \xiv^x(s), 1\rangle \right\rvert &= \left\lvert \int_{\X}  1 \left[\xi_v^x(s)(dz) -\bar{\xi}^x_v(s)(dz)  \right] \right\rvert \leq \left\lvert  \int_{ E \times \X}  \indic{x}(x\myprime) \, \left( \xi_v(s) -\bar{\xi}_v(s)\right)(dx\myprime,dz) \right\rvert \nonumber \\
&\leq  \sup_{\phi \leq 1}  \,  \lvert \langle \xi_v(s) - \bar{\xi}_v(s), \phi \rangle \rvert ,
\end{align*}
and by Remark \eqref{RemContrlReg1}:
\begin{align*}
 \int_0^t  \left\lvert\int_{\Vp}  \left[ \langle \bar{\xi}_v^x(s), 1\rangle- \langle \xiv^x(s), 1\rangle \right]  \phi_v(x,z) \,\bar{\xi}_v(s)(dx,dz) \right\rvert ds \leq e^{Ct} \, \max_{\alpha} \left( \langle \xi_\alpha(0), 1\rangle\right) \, \int_0^t  \sup_{\phi \leq 1}  \,  \lvert \langle \xi_v(s) - \bar{\xi}_v(s), \phi \rangle \rvert  \, ds.
\end{align*}
Additionally, also by Remark \eqref{RemContrlReg1} we have:
\begin{align*}
 \int_0^t  \left\lvert\int_{\Vp}  \langle \xiv^x(s), 1\rangle \,  \phiv(x,z) \left[\xi_v(s)(dx,dz)-\bar{\xi}_v(s)(dx,dz)\right] \right\rvert ds  \leq e^{Ct} \, \max_{\alpha} \left( \langle \xi_\alpha(0), 1\rangle\right) \, \int_0^t  \sup_{\phi \leq 1}  \,  \lvert \langle \xi_v(s) - \bar{\xi}_v(s), \phi \rangle \rvert  \, ds.
\end{align*}
Now we deal with the non-linear terms representing interaction with free vectors.  First, by Assumption \ref{AssPherates}, we have:
\begin{equation*}
  \int_{ \mathcal{D}} \beta(s,y,x,\langle \xi_v^x(s), 1\rangle,z) \,  \left[\xiu(s)(dy)-\bar{\xi}_u(s)(dy)\right]  \leq \bar{\beta} \,    \sup_{\phi \leq 1}  \,  \lvert \langle \xiu(s) - \bar{\xi}_u(s), \phi \rangle \rvert, 
\end{equation*}
and by Remark \eqref{RemContrlReg1} we deduce:
\begin{align*}
\int_0^t  &\left\lvert \int_{\Vp \times \D} \beta(s,y,x,\langle \xi_v^x(s), 1\rangle,z) \,  \left[\xiu(s)(dy)-\bar{\xi}_u(s)(dy)\right] \, \phiv(x,z) \, \xi_v(s)(dx,dz)  \right\rvert \, ds \nonumber \\
&\leq \bar{\beta} \, e^{Ct} \, \max_{\alpha} \left( \langle \xi_\alpha(0), 1\rangle\right) \, \int_0^t  \sup_{\phi \leq 1}  \,  \lvert \langle \xiu(s) - \bar{\xi}_u(s), \phi \rangle \rvert  \, ds.
\end{align*}
Analogously we also obtain:
\begin{align*}
\int_0^t  &\left\lvert \int_{\Vp \times \mathcal{D}} \beta(s,y,x,\langle \bar{\xi}_v^x(s), 1\rangle,z) \, \left[ \xiu(s)(dy) -  \bar{\xiu}(s)(dy) \right]  \, \phiv(x,z) \, \bar{\xi}_v(s)(dx,dz)  \right\rvert \, ds \nonumber \\
&\leq \bar{\beta} \, e^{Ct} \, \max_{\alpha} \left( \langle \xi_\alpha(0), 1\rangle\right) \, \int_0^t  \sup_{\phi \leq 1}  \,  \lvert \langle \xi_u(s) - \bar{\xi}_u(s), \phi \rangle \rvert  \, ds.
\end{align*}
The Lipschitz continuity of $\beta$ also implies the upper bound:
\begin{align*}
 \int_0^t  &\left\lvert \int_{\Vp \times \mathcal{D}} \left[ \beta(s,y,\langle \xi_v^x(s), 1\rangle,z)  -\beta(s,y,x,\langle \bar{\xi}_v^x(s), 1\rangle,z)  \right] \, \xi_u(s)(dy) \,  \phiv(x,z) \, \bar{\xi}_v(s)(dx,dz)  \right\rvert \, ds \nonumber \\
&\hspace{2cm}\leq \bar{\beta} \, e^{Ct} \, \left[ \max_{\alpha} \left( \langle \xi_\alpha(0), 1\rangle\right) \right]^2 \, \int_0^t  \sup_{\phi \leq 1}  \,  \lvert \langle \xi_u(s) - \bar{\xi}_u(s), \phi \rangle \rvert  \, ds.
\end{align*}
Altogether implies:
\begin{align*}
\left\lvert \langle \xi_v(t) -\bar{\xi}_v(t) , \phi_v\rangle \right\rvert &\leq C \left( \int_0^t \sum_{\alpha} \, \sup_{\phi \leq 1}  \,  \lvert \langle \xi_\alpha(s) - \bar{\xi}_\alpha(s), \phi \rangle \rvert  \, ds. \right)
\end{align*}
For $\alpha \in \lbrace u, c \rbrace$, we cannot proceed in the same way to bound 
\begin{align*}
\left\lvert \langle \xi_v(t) -\bar{\xi}_v(t) , \phi_\alpha \rangle \right\rvert. 
\end{align*}
 The reason for the additional difficulty is that for a generic measurable $\phi_\alpha$, the action of the generator $\L^\alpha$ is not necessarily well-defined. We will avoid this issue by using an approach similar to the one given in the proof of Theorem 4.2 in \cite{champagnat_invasion_2007}.\\

For $\alpha \in \lbrace u, c \rbrace $, we consider the semi-groups $P_\alpha(t)$ corresponding to the diffusion process with generator $\L^\alpha$. Let us fix a function $\phi_\alpha \in L^\infty(\V_\alpha)$ with $\norm{\phi_\alpha } \leq 1$, take $t \in [0,T]$, and define the following function:
\begin{equation*}
  \phi_\alpha(s, y_\alpha) = P_\alpha (t-s) \phi(y_\alpha) \quad \forall y_\alpha \in \V_\alpha. 
\end{equation*}
By construction $\phi_\alpha(s, y_\alpha)$ is a solution of the boundary problem:
\begin{align*}
    \partial_s \phi_\alpha(s, y_\alpha) &- \L^\alpha \phi_\alpha(s, y_\alpha) = 0 \quad \text{ on } [0,T] \times \V_\alpha  \nonumber \\
    \nabla \phi_\alpha \cdot \bar{n}(s, y_\alpha) &= 0 \quad \text{ on } [0,T] \times \partial \V_\alpha  \nonumber \\
    \lim_{ s \to t} \phi_\alpha(s, y_\alpha) &=  \phi_\alpha( y_\alpha) \quad \text{ on } \V_\alpha \nonumber
\end{align*}
where, to save some space, we have made a slight abuse of notation by using $\partial \Vc = \partial \D \times \X$.\\

Using the weak time-space formulation given in \eqref{WEAKST2} and \ref{WEAKST3}, we obtain:
\begin{align*}
\langle \xiu(t) , \phiu \rangle &=  \langle \xiu(0) , P_u(t) \phiu \rangle -  \int_0^t  \int_{\Vp \times \D} \beta(s,y,x,\langle \xi_v^x(s), 1\rangle,z) P_u(t-s)\phiu(y) \xiv(s)(dx,dz) \,  \xiu(s)(dy) \, ds \nonumber \\
&+\int_0^t   \int_{\Vc \times E} \eta(s,y,x,z) P_u(t-s)\phiu(y)) \xic(s)(dy,dz) \,  dx \,   ds + \int_0^t \int_{\Vc} \gamma(z) P_u(t-s)\phiu(y) \, \xic(s)(dy,dz) ds,  \nn 
\end{align*}
and 
\begin{align*}
\langle \xic(t) , \phic\rangle&= \langle \xic(0) , P_c(t)\phic\rangle + \int_0^t  \int_{\Vp \times \mathcal{D}} \beta(s,y,x,\langle \xi_v^x(s), 1\rangle,z) P_c(t-s) \phic(y,z)  \xiv(s)(dx,dz) \,  \xiu(s)(dy) \, ds \nonumber \\
&-  \int_0^t   \int_{\Vc \times E} \eta(s,y,x,z) P_c(t-s) \phic(y,z)   \xic(s)(dy,dz) \,  dx \,   ds -\int_0^t \int_{\Vc} \gamma(z) P_c(t-s) \phic(y,z)  \, \xic(s)(dy,dz) \,  ds. \nn
\end{align*}
Analogous estimates to those used for the case $\alpha=v$, give
\begin{align*}
\left\lvert \langle \xiu(t) -\bar{\xi}_u(t) , \phiu\rangle \right\rvert &\leq C \left( \int_0^t \sum_{\alpha} \, \sup_{\phi \leq 1}  \,  \lvert \langle \xi_\alpha(s) - \bar{\xi}_\alpha(s), \phi_\alpha \rangle \rvert  \, ds \right) \nn 
\end{align*}
and 
\begin{align*}
\left\lvert \langle \xic(t) -\bar{\xi}_c(t) , \phic\rangle \right\rvert &\leq C \left( \int_0^t \sum_{\alpha} \, \sup_{\phi \leq 1}  \,  \lvert \langle \xi_\alpha(s) - \bar{\xi}_\alpha(s), \phi_\alpha \rangle \rvert  \, ds \right) \nn 
\end{align*}
where we have picked the biggest constant $C$ of all cases. Taking first the sup on the LHS of each case, and then summation gives the inequality:
\begin{align*}
\sum_{\alpha} \, \sup_{\phi \leq 1}  \,  \left\lvert \langle \xi_\alpha(t) -\bar{\xi}_\alpha(t) , \phi_\alpha\rangle \right\rvert &\leq C \left( \int_0^t \sum_{\alpha} \, \sup_{\phi \leq 1}  \,  \lvert \langle \xi_\alpha(s) - \bar{\xi}_\alpha(s), \phi_\alpha \rangle \rvert  \, ds \right), \nn 
\end{align*}
and  by Gronwall's lemma we conclude uniqueness.

\subsubsection*{Uniform estimates}
Let us first remark that, under Assumption \ref{HydroAssump}, we have the following estimate:
\begin{equation}\label{unifoverKontrol}
    \sup_{K \in \N} \E \left[ \sup_{ t \in [0,T]} \langle \nu_\alpha^{(K)}(t), 1 \rangle^3 \right] < +\infty, 
\end{equation}
for every $\alpha \in \lbrace v,u,c \rbrace$.\\

To see that this is the case, notice that by replicating the  procedure of the proof of Proposition \ref{Propcontrolp}, with $p=3$, we can obtain an analogous expression to \eqref{GronwApp}, namely
\begin{equation*}
 \E \left[ \sup_{t \in [0,T \wedge \tau_n^{(K)}] } \langle  \nu_\alpha^{(K)}(t),1 \rangle^3 \right] \leq   C e^{BT}
\end{equation*}
where the constant $C$ is independent of $K$ and $n$. Just as in Proposition \ref{Propcontrolp} we conclude that $\tau_n^{(K)} \to \infty$, as $K \to \infty$ for some well chosen $n$, and obtain \eqref{unifoverKontrol} by Fatou's lemma.\\

\begin{remark}
As a consequence of estimate \eqref{unifoverKontrol} we have that for every $\alpha \in \lbrace v,u,c \rbrace$
\begin{equation}\label{unifoverKontrol3}
    \sup_{K \in \N} \E \left( \sup_{ t \in [0,T]} \langle \nu_\alpha^{(K)}(t), \phi_\alpha \rangle^3 \right) < \infty 
\end{equation}
for all $\phi_\alpha$ as in Definition \ref{correncod}.
\end{remark}

\subsubsection*{Tightness}
 The goal of this section is to verify that, for each $\alpha \in \lbrace v, u, c \rbrace$, the sequence of laws $Q_\alpha^K$ of the processes $\nu_\alpha^{(K)}$ are uniformly tight in $\P(\D([0,T],\M_F^\alpha))$. Here we specialize to the case $\lambda=1$. The other case is simpler for the vector populations, and for the virus population it is virtually done the same way as for $\lambda=1$. Let us first introduce the following lemma:
 
 \begin{lemma}\label{lemmakmart}
Under Assumption \ref{Assbetakreg2} we have the following càdlàg martingales:
\begin{enumerate}
\item For the virus population:
\begin{align}\label{kMartH}
M_K^{\phiv}(t) &= \langle \nuv^{(K)}(t) , \phiv\rangle- \langle \nuv^{(K)}(0) , \phiv\rangle - \int_0^t \int_{\Vp} b(x,z) \phiv(x,z) \nuv^{(K)}(s)(dx,dz) \,  ds \nonumber \\
&\hspace{-1cm}+ \mu \int_0^t \int_{\Vp} b(x,z)\left[ \phiv(x,z) -\int_{ \X} m(z,e) \phiv(x,e) \, de  \right] \nuv^{(K)}(s)(dx,dz) \,  ds \nonumber \\
&\hspace{-1cm}+  \int_0^t  \int_{\Vp} \left( d(z) + c \, \langle (\nu_v^{(K)})^{x}, 1\rangle \right) \phiv(x,z) \nuv^{(K)}(s)(dx,dz) ds \nonumber \\
&+ \int_0^t \int_{\Vp \times \D} \beta(s,y,x, \langle (\nu_v^{(K)})^{x}, 1\rangle,z) \, \phiv(x,z) \nuv^{(K)}(s)(dx,dz) \,  \nuu^{(K)}(s)(dy)  ds \nonumber \\
&- \int_0^t   \int_{\Vc \times E} \eta(s,y,x,z) \phiv(x,z)  \,  \nuc^{(K)}(s)(dy,dz) \, dx \,   ds 
\end{align}
is a càdlàg martingale with predictable quadratic variation
\begin{align}\label{kMartHQV}
\langle M_K^{\phiv} \rangle_t &=   \frac{(1-\mu)}{K} \int_0^t \int_{\Vp} b(x,z) \, \phiv(x,z)^2 \, \nuv^{(K)}(s)(dx,dz) \,  ds \nonumber \\
&+  \frac{\mu}{K}  \int_0^t \int_{\Vp \times \X} b(x,z) m(z,e)  \,  \phiv(x,e)^2  \, de  \, \nu_v^{(K)}(s)(dx,dz) \,  ds \nonumber \\
&+ \frac{1}{K} \int_0^t  \int_{\Vp} \left( d(z) + c \, \langle (\nuv^{(K)})^{x}, 1\rangle \right)  \,   \phiv(x,z)^2 \, \nuv^{(K)}(s)(dx,dz)  \nonumber \\
&+ \frac{1}{K} \int_0^t  \int_{\Vp \times \D} \beta(s,y,x, \langle (\nu_v^{(K)})^{x}, 1\rangle,z) \,   \phiv(x,z)^2 \, \nuv^{(K)}(s)(dx,dz) \,  \nuu^{(K)}(s)(dy)  ds \nonumber \\
& + \frac{1}{K} \int_0^t   \int_{\Vc \times E} \eta(s,y,x,z) \,  \phiv(x,z)^2   \,  \nuc^{(K)}(s)(dy,dz) \, dx \,   ds.
\end{align}
\item For the uncharged vector population:
\begin{align}\label{kMartF}
M_K^{\phiu}(t) &= \langle \nuu^{(K)}(t) , \phiu\rangle- \langle \nuu^{(K)}(0) , \phiu\rangle - \int_0^t \int_{\D} \L^u \phiu(y) \, \nuu^{(K)}(s)(dy) \, ds \nonumber \\
&+ \int_0^t  \int_{\Vp \times \D} \beta(s,y,x, \langle (\nu_v^{(K)})^{x}, 1\rangle,z) \, \phiu(y) \nuv^{(K)}(s)(dx,dz) \,  \nuu^{(K)}(s)(dy) \, ds \nonumber \\
& - \int_0^t   \int_{\Vc \times E} \eta(s,y,x,z) \phiu(y) \nuc^{(K)}(s)(dy,dz) \,  dx \,   ds -\int_0^t \int_{\Vc} \gamma(z) \phiu(y) \, \nuc^{(K)}(s)(dy,dz) ds 
\end{align}
is a càdlàg martingale with predictable  quadratic variation
\begin{align}\label{kMartFQV}
\langle M_K^{\phiu} \rangle_t &=  \frac{1}{K} \int_0^t \int_{\D} \sigma^u(y)^2  \, |\nabla \phiu(y) |^2 \, \nuu^{(K)}(s)(dy) \, ds \nn \\
&+ \frac{1}{K}\int_0^t  \int_{\Vp \times \D} \beta(s,y,x, \langle (\nu_v^{(K)})^{x}, 1\rangle,z) \,   \phiu(y)^2 \, \nuv^{(K)}(s)(dx,dz) \,  \nuu^{(K)}(s)(dy)  ds \nonumber \\
&+ \frac{1}{K}  \int_0^t  \int_{\Vc \times E} \eta(s,y,x,z) \phiu(y)^2 \nuc^{(K)}(s)(dy,dz) \,  dx \,   ds +\frac{1}{K}\int_0^t \int_{\Vc} \gamma(z) \phiu(y)^2 \, \nuc^{(K)}(s)(dy,dz) \, ds. 
\end{align}
\item For the charged-vector population:
\begin{align}\label{kMartG}
M_K^{\phic}(t) &= \langle \nuc^{(K)}(t) , \phic\rangle- \langle \nuc^{(K)}(0) , \phic\rangle -\int_0^t \int_{\Vc} \L^c \phic(y,z) \, \nuc^{(K)}(s)(dy,dz) \, ds \nonumber \\
&-  \int_0^t  \int_{\Vp \times \D} \beta(s,y,x, \langle (\nu_v^{(K)})^{x}, 1\rangle,z) \, \phic(y,z) \nuv^{(K)}(s)(dx,dz) \,  \nuu^{(K)}(s)(dy) \, ds \nonumber \\ 
& + \int_0^t   \int_{\Vc \times E} \eta(s,y,x,z) \phic(y,z) \nuc^{(K)}(s)(dy,dz) \,  dx \,   ds +  \int_0^t \int_{\Vc} \gamma(z) \phic(y,z) \, \nuc^{(K)}(s)(dy,dz)   ds
\end{align}
is a càdlàg martingale with quadratic variation
\begin{align}\label{kMartGQV}
\langle M_K^{\phic} \rangle_t  &=  \frac{1}{K}\int_0^t \int_{\Vc} \sigma^c(y)^2 | \nabla_y \phic(y,z) |^2 \, \nuc^{(K)}(s)(dy,dz) \, ds \nonumber \\
&+  \frac{1}{K}\int_0^t  \int_{\Vp \times \D} \beta(s,y,x, \langle (\nu_v^{(K)})^{x}, 1\rangle,z) \, \phic(y,z)^2 \nuv^{(K)}(s)(dx,dz) \,  \nuu^{(K)}(s)(dy) \, ds \nonumber \\ 
& + \frac{1}{K}  \int_0^t   \int_{\Vc \times E} \eta(s,y,x,z) \phic(y,z)^2 \nuc^{(K)}(s)(dw,du) \,  dx \,   ds +\frac{1}{K} \int_0^t \int_{\Vc} \gamma(y) \phic(y,z)^2 \, \nuc^{(K)}(s)(dy,dz)  \, ds.
\end{align}
\end{enumerate}
\end{lemma}

\begin{proof}
The proof of this lemma is a replica of Proposition \ref{MartingalesProp} taking into consideration the extra powers of $K$ appearing in the generators.
\end{proof}
 
 In order to show tightness, and given the control given by Proposition \ref{Propcontrolp}, it is enough  (see for example Corollary 7.4  in \cite{ethier2009markov}) to show that for each $\alpha \in \lbrace v, u, c \rbrace $ and $\varepsilon >0$ there exists $\delta>0$ and  $\rho >0$, such that:
\begin{equation}
    \sup_{ K \in \N} \P\left(  \omega(M_K^\alpha, \delta, T) > \rho \right) \leq \varepsilon, \text{ and } \sup_{ K \in \N} \P\left(  \omega( \langle M_K^\alpha \rangle, \delta, T) > \rho \right) \leq \varepsilon,
\end{equation}
where $M_K^{(\alpha)}$ are given as in Lemma \ref{lemmakmart} with $\phi_\alpha=1$, and for $\delta>0$, $\omega(X,\delta,T)$ denotes the standard modulus of continuity:
\begin{equation*}
\omega(X,\delta,T)= \sup_{\substack{t,s \in [0,T]\\ |t-s| \leq \delta}} |X(t)-X(s)|.
\end{equation*}
Let us show how to control de modulus of continuity for the predictable quadratic variation process $\langle M_K^\alpha \rangle$.  
Let $\delta >0$, and consider stopping times $\tau_K, \tau\myprime_K$ such that:
\[
0 \leq \tau_K \leq \tau\myprime_K \leq \tau_K + \delta \leq T
\]
By Doob's inequality, we have:
\begin{align}\label{Doobsapp}
\E \left( \langle M_K^{\alpha}  \rangle_{\tau_K} -   \langle M_K^{\alpha}  \rangle_{\tau\myprime_K}  \right) \leq \sum_{\alpha\myprime} \E \left( C \int_{\tau_K}^{\tau_K + \delta} \left( \langle \nu_{\alpha\myprime}^{(K)}(s), 1 \rangle + \langle \nu_{\alpha\myprime}^{(K)}(s), 1 \rangle^2 \right) \, ds \, \right)
\end{align}
where we used the fact that the quadratic variation processes are given explicitly in \eqref{kMartHQV}, \eqref{kMartFQV} and \eqref{kMartGQV}, together with the technique used in the proof of Proposition \ref{ControlJumpRate}. By \eqref{unifoverKontrol}, and a redefinition of the constant $C$, we obtain:
\begin{align*}
\E \left( \langle M_K^{\alpha}  \rangle_{\tau_K} -   \langle M_K^{\alpha}  \rangle_{\tau\myprime_K}  \right) \leq C \delta 
\end{align*}

We conclude the proof of uniform tightness of the sequence $Q_\alpha^{K}$.

\subsection{Proof of Theorem \ref{IDElambdaequal}}\label{firstproof}
We start by noticing that since both populations are comparable in order, i.e. $\lambda=1$, the time rescaling of the diffusive generators for the population of vectors reduces to normal speed. In this section we will mainly follow the approach given in the proof of Theorem 4.2 from \cite{champagnat_invasion_2007}. This approach requires to control the quadratic variation of  some relevant martingales and show that, as consequence of the vanishing quadratic variation, the martingales themselves also vanish as $K \to \infty$. Lemma \ref{lemmakmart} introduced the above mentioned martingales and their quadratic variations.

\subsubsection{Characterization of limit points for Theorem \ref{IDElambdaequal}}\label{CharactLamdaequal}
By tightness we know that for $\alpha \in \lbrace v, u, c \rbrace$ the sequence of measures $Q_\alpha^K$  contains at least a convergent sub-sequence. Let us denote by $Q_\alpha$ the limiting law in $\P(\D([0,T],\M_F(D_\alpha)))$ of any such sub-sequence, which by an abuse of notation we still denote by $Q_\alpha^K$. We first want to argue that every process $\xi_\alpha$ with law $Q_\alpha$ is almost surely strongly continuous. We can see that this is indeed the case from the following estimate:
\begin{equation*}
 \sup_{ t \in [0,T]} \lvert \langle \nu_\alpha^{(K)}(t), \phi_\alpha \rangle - \langle \nu_\alpha^{(K)}(t-), \phi_\alpha \rangle \rvert \leq \frac{\norm{\phi_\alpha}_\infty}{K}
\end{equation*}
which is true because at every jump event we either kill, create, or re-distribute one virus. In order to show that $Q_\alpha$ only charges the continuous processes we follow the lines of Step 5, page 54 of \cite{meleard_stochastic_2015}. Since, for each $\phi_\alpha$, the mapping $t \mapsto  \sup_{ t \in [0,T]} \lvert \langle \nu_\alpha(t), \phi_\alpha \rangle - \langle \nu_\alpha(t-), \phi_\alpha \rangle \rvert $ is continuous on $\D([0,T],\M_F(D_\alpha))$ we can deduce that $Q_\alpha$ only charges continuous processes from $[0,T]$ to $M_F(\D_\alpha)$ endowed with the vague topology. To extend this reasoning to the case of $M_F(\D_\alpha)$ being endowed with the weak topology we can follow the lines of Step 6, page 24, in \cite{champagnat_individual_2008}. \\

We now show that the limits $\xi_v, \xiu$ and $\xic$ indeed satisfy the system of IDE's. We then introduce the following quantities:
\begin{align*}
M_t^{v,h} (\xiv,\xiu,\xic) &:= \langle \xiv(t) , h\rangle   -\langle \xiv(0) , h\rangle -\int_0^t \int_{\Vp} b(x,z) h(x,z) \, \xiv(s)(dx,dz) \,  ds \nonumber \\
&+ \mu \int_0^t \int_{\Vp} b(x,z)\left[ h(x,z) -\int_{ \mathcal{X}} m(z,e) h(x,e) \, de  \right] \xiv(s)(dx,dz) \,  ds \nonumber \\
&+  \int_0^t  \int_{\Vp}  \left[ d(z) + c \langle \xiv(s)^x, 1\rangle \right] h(x,z) \xiv(s)(dx,dz) ds \nonumber \\
&+ \int_0^t  \int_{\Vp \times \mathcal{D}} \beta(s,y,x,\langle \xiv(s)^x, 1\rangle,z) h(x,z) \, \xiv(s)(dx,dz) \,  \xiu(s)(dy) \, ds \nonumber \\
&- \int_0^t   \int_{\mathcal{I} \times E} \eta(s,y,x,z) h(x,z)  \,  \xic(s)(dy,dz) \, dx \,   ds,
\end{align*}
\begin{align*}
M_t^{u,f} (\xiv,\xiu,\xic) &:= \langle \xiu(t) , f\rangle -  \langle \xiu(0) , f\rangle -\int_0^t \int_{\mathcal{D}} \mathcal{L}^u f(y) \, \xiu(s)(dy) \, ds \nonumber \\
&+  \int_0^t  \int_{\Vp \times \mathcal{D}} \beta(s,y,x,\langle \xiv(s)^x, 1\rangle,z) f(y) \xiv(s)(dx,dz) \,  \xiu(s)(dy) \, ds \nonumber \\
&-\int_0^t   \int_{\Vc \times E} \eta(s,w,x,u) f(w) \xic(s)(dw,du) \,  dx \,   ds - \int_0^t \int_{\Vc} \gamma(u) f(w) \, \xic(s)(dw,du) ds, 
\end{align*}
and
\begin{align*}
M_t^{c,g}  (\xiv,\xiu,\xic) &:= \langle \xic(t) , g\rangle- \langle \xic(0) , g\rangle -\int_0^t \int_{\Vc} \mathcal{L}^c g(w,u) \, \xic(s)(dw,du) \, ds \nonumber \\
&- \int_0^t  \int_{\Vc \times \mathcal{D}} \beta(s,y,x,\langle \xiv(s)^x, 1\rangle,z) g(y,z) \xiv(s)(dx,dz) \,  \xiu(s)(dy) \, ds \nonumber \\
&+  \int_0^t   \int_{\Vc \times E} \eta(s,w,x,u) g(w,u) \xic(s)(dw,du) \,  dx \,   ds +\int_0^t \int_{\Vc} \gamma(u) g(w,u) \, \xic(s)(dw,du) \,  ds.
\end{align*}
In order to verify that the processes indeed satisfy the system of IDE's, our goal is to show that we have:
\begin{align*}
    \E \left[ \lvert M_t^{\alpha,\phi_\alpha}  (\xiv,\xiu,\xic)\rvert  \right] &= 0, 
\end{align*}
for all $\alpha \in \lbrace v, u, c \rbrace.$\\

Notice that the following relations hold:
\begin{equation*}
M_K^{\phi_\alpha}(t)  =  M_t^{\alpha, \phi_\alpha}  (\nuv^{(K)}(t),\nuu^{(K)}(t),\nuc^{(K)}(t)), \qquad
\E \left( \lvert M_K^{\phi_\alpha}(t)  \rvert^2  \right) =     \E \left( \langle M_K^{\phi_\alpha} \rangle_t\right), 
\end{equation*}
for all $\alpha \in \lbrace v, u, c \rbrace.$ By Assumption \ref{AssPherates} , Proposition \ref{lemmakmart}, and similar estimates than the ones used in \eqref{Doobsapp}, we have:
\begin{align*}
\E \left( \lvert M_K^{\phi_\alpha}(t)  \rvert^2  \right) &\leq  \frac{C_{\phi_\alpha,t}^\alpha}{K}, 
\end{align*}
where the constants $C_{\phi_\alpha,t}^\alpha$, might depend on time and the uniform norm of the functions $\phi_\alpha$, but do not depend on the scaling parameter $K$.\\

To conclude we need to show:
\begin{align*}
  \lim_{K \to \infty}  \E \left( \lvert M_t^{\alpha, \phi_\alpha} (\nuv^{(K)}(t),\nuu^{(K)}(t),\nuc^{(K)}(t) \rvert  \right) &= \E \left( \lvert M_t^{\alpha, \phi_\alpha} (\xiv, \xiu,\xic) \rvert  \right), 
\end{align*}
along any convergent sub-sequence of  $\lbrace (\nuv^{(K)}(t),\nuu^{(K)}(t),\nuc^{(K)}(t)) \rbrace_{K \geq 1}$. This is true if we are able to show uniform integrability for each collection $\lbrace M_t^{\alpha, \phi_\alpha}  (\nuv^{(K)}(t),\nuu^{(K)}(t),\nuc^{(K)}(t)) \rbrace_{K \geq 1}$, for all $\alpha \in \lbrace v, u, c \rbrace.$ \\

In order to do so, let $\alpha \in \lbrace v, u, c \rbrace$, and let us show that it is indeed the case that $\lbrace M_t^\alpha  (\nuv,\nuu,\nuc) \rbrace_{K \geq 1}$ is uniformly integrable. First notice, using Assumption \ref{AssPherates} and uniform estimates over the parameters and the test functions, that for any $(\nu_1,\nu_2,\nu_3)$ we have:
\begin{align*}
\lvert M_t^\alpha  (\nu_1,\nu_2,\nu_3) \rvert  &\leq C(t,\phi_v,\phi_S,\phi_I) \sup_{t \in [0,T]} \left( 1 + \langle \nu_1, 1\rangle^2 +  \langle \nu_2, 1\rangle^2 + \langle \nu_3, 1\rangle^2 \right),
\end{align*}
and by \eqref{unifoverKontrol3} we have uniform integrability. This finishes the proof.

\subsection{Proof of Theorem \ref{IDElambdaless}}\label{secondproof}
Given the assumptions on Theorem \ref{IDElambdaless}, it is not straightforward to adapt the proof of characterization of limit points of Theorem \ref{IDElambdaequal} to this case. The main difficulty is that we cannot control the quadratic variations \eqref{kMartFQV} and \eqref{kMartGQV} as directly as in the case $\lambda=1$. However, we can make use of the averaging principle for slow-fast systems introduced in \cite{kurtz1992averaging}, and avoid the need to directly control those quadratic variations. We postpone the proof of Theorem \ref{IDElambdaless} for Section \ref{prfIDElambdaless}. First we introduce the context of Kurtz's averaging principle for martingale problems.

\subsubsection{Averaging for martingale problems}

The proof of  Theorem \ref{IDElambdaless}  is based on Theorem 2.1 from \cite{kurtz1992averaging}. For the sake of readability we will now include this theorem without proof. We refer the reader to \cite{kurtz1992averaging} for relevant definitions and a proof of the theorem.

\begin{theorem}[T. Kurtz, 1992]\label{AverageThm}
Let $\s_1$ and $\s_2$ be complete separable metric spaces, and set $\s = \s_1 \times \s_2$. For each $K$, let $\lbrace (X_K, Y_K) \rbrace$ be a stochastic process with sample paths in $D([0,\infty), \s)$ adapted to a filtration $\lbrace \f_t^K \rbrace$. Assume that $\lbrace X_K \rbrace$ satisfies the compact containment condition, that is, for each $\eps >0$ and $T>0$, there exists a compact set $\C \in \s_1$ such that
\begin{equation}\label{compccont}
    \inf_{K} \P \left[ X_K(t) \in \C, t \leq T \right] \geq 1-\eps,
\end{equation}
and assume that $\lbrace Y_K(t): t \geq 0,K \in \N \rbrace$ is relatively compact (for every $t \geq 0$, and as a collection of $\s_2$-valued random variables). Suppose the existence of an operator $A: \D(A) \subset \C_b(\s_1) \to \C_b(\s)$ such that, for every $f \in \D(A)$, there exists a process $\epsilon_K^f$ for which
\begin{equation*}
    f(X_K(t)) - \int_0^t A f(X_K(s),Y_K(s)) ds + \epsilon_K^f(t)
\end{equation*}
is an $\f_t^K$-martingale. Let $\D(A)$ be dense in $\C_b(\s_1)$ in the topology of uniform convergence on compact sets. Suppose that for each $f \in \D(A)$ and each $T>0$, there exists $p>1$ such that
\begin{equation*}
    \sup_{K} \E \left[ \int_0^T \mid A f(X_K(t),Y_K(t)) \mid^p dt \right] < \infty,
    \end{equation*}
and 
\begin{equation}\label{epsK}
    \lim_{K \to \infty} \E \left[ \sup_{t \leq T}  \mid \epsilon_K^f(t) \mid \right] = 0.
\end{equation}
Let $\Gamma_K$ be the $l_m(\s_2)$-valued random variable given by
\begin{equation*}
    \Gamma_K ([0,t] \times B) = \int_0^t \indic{ B}( Y_K(s)) ds.
\end{equation*}
Then $\lbrace X_K, \Gamma_K \rbrace$ is relatively compact in $D([0,\infty),\s_1) \times l_m(\s_2)$ and for any limit point $\left( X, \Gamma \right)$ there exists a filtration $\lbrace \g_t  \rbrace$ such that
\begin{equation*}
    f(X(t)) - \int_0^t \int_{\s_2} A f(X(s),y) \Gamma(ds \times dy)
\end{equation*}
is a $\g_t$-martingale for each $f \in \D(A)$.
\end{theorem}

The way we will apply this theorem in the proof Theorem \ref{IDElambdaless} is mainly based on the following example from the same source:

\subsubsection*{Example 2.3 of \cite{kurtz1992averaging}:}
Suppose the existence of an operator $B: \D(B) \subset \C_b(\s_2) \to \C_b(\s)$ such that there exists $\mu_K \in \R$, such that for all $g \in \D(B)$, the quantity
\begin{equation*}
    g(Y_K(t)) - \int_0^t \mu_K Bg(X_K(s),Y_K(s)) ds + \delta_K^g(t)
\end{equation*}
is an $\f_t^K$-martingale, $\mu_K \to \infty$, and for each $T>0$
\begin{equation*}
    \lim_{K \to \infty} \E \left[ \sup_{ t \in [0,T]} \frac{1}{\mu_K} \mid\delta_K^g(t) \mid  \right] = 0
\end{equation*}
Then, under the assumptions of Theorem \ref{AverageThm}, it follows that
\begin{equation}\label{FastMart}
    \int_{[0,t] \times \s_2} Bg(X(s),y) \Gamma(ds \times dy)
\end{equation}
is a martingale. But \eqref{FastMart} is continuous and of bounded variation and therefore constant. As a consequence, for each $g \in \D(B)$, with probability one
\begin{equation}\label{FastMart2}
    \int_{[0,t] \times \s_2} Bg(X(s),y) \Gamma(ds \times dy) = 0
\end{equation}
for all $t>0$. Then by Lemma 1.4 of \cite{kurtz1992averaging} there exists a $\P(\s_2)$-valued process $\lbrace \pi_t \rbrace$ such that:
\begin{equation*}
    \int_{[0,t] \times \s_2} h(s,y) \Gamma(ds \times dy) = \int_0^t \left( \int_{\s_2} h(s,y) \pi_s(dy) \right) \Gamma(ds \times \s_2) 
\end{equation*}
with probability one, and for all Borel-measurable $h$ in $[0,\infty) \times \s_2$.\\

At this point, T. Kurtz assumes the existence of a countable subset $\hat{D} \subset \D(B)$ such that:
\begin{equation}\label{closurecondition}
   \closure[0]{ \lbrace (g,Bg): g \in \hat{D} \rbrace} = \closure[0]{\lbrace (g,Bg): g \in \D(B) \rbrace}
\end{equation}
where both closures are taken in $\C_b(\s_2) \times \C_b(\s_1 \times \s_2)$ with respect to the topology of uniform convergence.\\

This gives sufficient conditions such that $\Gamma \in l_m(\s_2)$, and as a consequence \eqref{FastMart2} can be written as:
\begin{equation*}
    \int_0^t \left( \int_{\s_2} Bg(X(s),y) \pi_s(dy) \right) \, ds = 0
\end{equation*}
for all $t$ a.s., and hence
\begin{equation}\label{FastMart4}
   \int_{\s_2} Bg(X(s),y) \pi_s(dy)  = 0
\end{equation}
almost everywhere Lebesgue almost surely.\\

Under additional assumptions, \eqref{FastMart4} implies that the measure $\pi_s$ is stationary for the process with  generator $B_{X(s)}$ given by:
\begin{equation*}
    B_{X(s)}g(y) = Bg(X(s),y)
\end{equation*}
for all $g \in \D(B_{X(s)})$.

\subsubsection{Characterization of limit points for Theorem \ref{IDElambdaless}}\label{prfIDElambdaless}

Let us define $X_K(t)$ and $Y_K(t)$ by:
\begin{align*} 
     X_K(t) &:=  \nu_v^{(K)}(t) \in \M_F^v\\ 
     Y_K(t) &:=  ( \nuu^{(K)} (t), \nuc^{(K)}(t) ) \in \M_F(\D) \times \M_F(\Vc) =: \M_F^{\text{vec}}
\end{align*}
where we think of $Y_K(t)$ as a measure on $\D \times \Vc$.\\

In order to be able to apply Theorem \ref{AverageThm} we need to proceed as follows:
\begin{enumerate}
    \item Verify that $\lbrace \nu_v^{(K)}(t) : t \geq 0\rbrace$ satisfies the compact containment condition.
    \item Verify that $\lbrace (\nuu^{(K)} (t), \nuc^{(K)}(t)): t \geq 0\rbrace$ is relatively compact.
    \item Identify the operator $A$, and related processes of Kurt's example.
\end{enumerate}

\subsubsection*{Compact containment}
Let $\eps >0$ and $T>0$ be fixed. Consider the set $ \C(T,\eps) \subset \M_F^v$ given by:
\begin{equation*}
    \C(T,\eps) = \lbrace \mu \in \M_p^v : \mu(\Vp) < \frac{C_T}{\eps}  \rbrace 
\end{equation*}
where $C_T$ is equal to the RHS of \eqref{GronwApp} for $p=1$.\\

\begin{remark}
Notice that, as a consequence of  $\M_F(\Vp)$ being Polish, the set $\C(T,\eps)$ is sequentially compact and hence compact.
\end{remark}
By Markov's inequality we have:
\begin{align*}
   \hspace{-4cm} \P \left[ \nu_v^K(t) \in \C(T,\eps) : t \in [0,T]  \right] &=  1-  \P \left[ \exists t \in [0,T] : \nu_v^K(t) \notin \C(T,\eps)   \right] \nn \\
    &\geq   1-  \P \left[ \sup_{t \in [0,T]} \langle \nu_v^K(t), 1\rangle  \geq  \frac{C_T}{\eps}   \right] \nn \\
    &\geq 1 - \frac{\eps}{C_T}\E \left[ \sup_{t \in [0,T]} \langle \nu_v^K(t) ,1 \rangle \right] \geq  1- \eps 
\end{align*}
which shows \eqref{compccont}.

\subsubsection*{Relative compactness for the fast system}
We want to to show that  for every $t \geq 0$   the sequence $\lbrace ( \nuu^{(K)} (t), \nuc^{(K)}(t) ) :  K \in \N \rbrace$ is relatively compact as a sequence of $\M_F^{\text{vec}}$-valued random variables. By Corollary 1.2 in \cite{kurtz1992averaging} it is enough to show that 
\begin{equation}\label{RelComp1}
    \sup_{K \in \N} \E \left[ \langle (\nuu^{(K)}(t), \nuc^{(K)}(t) ), 1 \rangle  \right] < \infty  
\end{equation}
and that for each $\eps >0$, there exists a compact $\C \subseteq \bar{D} \times \bar{D} \times \X$ such that:
\begin{equation*}
    \limsup_{K \to \infty} \E \left[ \langle (\nuu^{(K)},1  (t), \nuc^{(K)}(t) ), 1_{\C^c} \rangle > \eps \right] \leq \eps  
\end{equation*}
where $\C^c$ denotes the complement of $\C$.\\

Proposition \ref{Propcontrolp} gives \eqref{RelComp1} , and we conclude by choosing $\C = \bar{D} \times \bar{D} \times \X $.

\subsubsection*{Identification of the operators $A$ and $B$}
Before identifying the operators $A$ and $B$, we introduce the set of cylindrical functions on $M_F^\text{vec}$. These are functions that correspond to triplets:
\begin{equation}\label{CylOperaB}
\phivec_{uc} =\lbrace  \phiv=0, \phiu=\phiu, \phic=\phic \rbrace.
\end{equation}
i.e., functions of the form:
\begin{equation*}
    F_{\text{vec}}((\nuu,\nuc)):= F_{\phivec_{uc}}((\nuv,\nuu,\nuc)), 
\end{equation*}
where $F \in \C^2(\R^3)$ only depends in the last two coordinates, i.e., $F(x,y,z)=G(y,z)$ for some $G \in \C^2(\R^2)$, and the admissible triplet $\phivec_{uc}$ is given as in \eqref{CylOperaB} with $\phiu$ and $\phic$ given as in Definition \ref{correncod} .\\

\begin{proposition}\label{lemmakmartReg2}
Under Assumption \ref{Assbetakreg2} we have the following  martingale description of the virus population:
\begin{enumerate}
\item We have that:
\begin{align*}
M_K^{F_v}(t) &= F_v(\nu_v^{(K)}(t) )- K(1-\mu) \int_0^t \int_{\Vp} b(x,z) \left( F_v(\nu_v^{(K)}(s) + \frac{1}{K}\delta_{(x,z)}) - F_v(\nu_v^{(K)}(s) ) \right) \nu_v^{(K)}(s)(dx,dz) \,  ds \nonumber \\
&\hspace{-1cm}- K\mu \int_0^t \int_{\Vp \times \X} b(x,z)  m(z,e) \left( F_v(\nu_v^{(K)}(s) + \frac{1}{K}\delta_{(x,e)}) - F_v(\nu_v^{(K)}(s) ) \right) \, de  \, \nu_v^{(K)}(s)(dx,dz) \,  ds \nonumber \\
&\hspace{-1cm}+  K \int_0^t  \int_{\Vp} \left( d(z) + c \, \langle (\nu_v^{(K)})^{x}, 1\rangle \right) \left( F_v(\nu_v^{(K)}(s) - \frac{1}{K}\delta_{(x,z)}) - F_v(\nu_v^{(K)}(s) ) \right) \nu_v^{(K)}(s)(dx,dz) ds \nonumber \\
&\hspace{-1cm}+ \int_0^t  K\int_{\Vp \times \D} \beta(s,y,x,\langle (\nu_v^{(K)})^{x}, 1\rangle,z) \left( F_v(\nu_v^{(K)}(s) - \frac{1}{K}\delta_{(x,z)}) - F_v(\nu_v^{(K)}(s) ) \right) \nu_v^{(K)}(s)(dx,dz) \,  \nuu^{(K)}(s)(dy)  ds \nonumber \\
&\hspace{0cm}-K \int_0^t   \int_{\Vc \times E} \eta(s,y,x,z) \left( F_v(\nu_v^{(K)}(s) + \frac{1}{K}\delta_{(x,z)}) - F_v(\nu_v^{(K)}(s) ) \right) \,  \nuc^{(K)}(s)(dy,dz) \, dx \,   ds
\end{align*}
is a càdlàg martingale with predictable quadratic variation
\begin{align*}
\langle M_K^{F_v}\rangle_t &=  K(1-\mu) \int_0^t \int_{\Vp} b(x,z) \left( F_v(\nu_v^{(K)}(s) + \frac{1}{K}\delta_{(x,z)}) - F_v(\nu_v^{(K)}(s) ) \right)^2 \nu_v(s)(dx,dz) \,  ds \nonumber \\
&\hspace{-1cm}+ K\mu \int_0^t \int_{\Vp \times \X} b(x,z)  m(z,e) \left( F_v(\nu_v^{(K)}(s) + \frac{1}{K}\delta_{(x,e)}) - F_v(\nu_v^{(K)}(s) ) \right)^2 \, de  \, \nu_v^{(K)}(s)(dx,dz) \,  ds \nonumber \\
&\hspace{-1cm}+  K \int_0^t  \int_{\Vp} \left( d(z) + c \, \langle (\nu_v^{(K)})^{x}, 1\rangle \right) \left( F_v(\nu_v^{(K)}(s) - \frac{1}{K}\delta_{(x,z)}) - F_v(\nu_v^{(K)}(s) ) \right)^2 \nu_v^{(K)}(s)(dx,dz) ds \nonumber \\
&\hspace{-1cm}+ \int_0^t  K\int_{\Vp \times \D} \beta(s,y,x,\langle (\nu_v^{(K)})^{x}, 1\rangle,z) \left( F_v(\nu_v^{(K)}(s) - \frac{1}{K}\delta_{(x,z)}) - F_v(\nu_v^{(K)}(s) ) \right)^2 \nu_v^{(K)}(s)(dx,dz) \,  \nuu^{(K)}(s)(dy)  ds \nonumber \\
&\hspace{0cm}+K \int_0^t   \int_{\Vc \times E} \eta(s,y,x,z) \left( F_v(\nu_v^{(K)}(s) + \frac{1}{K}\delta_{(x,z)}) - F_v(\nu_v^{(K)}(s) ) \right)^2 \,  \nuc^{(K)}(s)(dy,dz) \, dx \,   ds, 
\end{align*}
 for any cylindrical function $F_v : \M p(E \times \X) \to \R$ of the form:
\begin{equation*}
    F_v(\nu) = F(\langle \nu, \phiv \rangle),
\end{equation*}
with $F \in \C^2(\R)$ and $\phiv$ measurable and bounded.
\end{enumerate}
\end{proposition}

\begin{proof}
The proof of this proposition is in the same spirit than the proof of  Proposition \ref{DynkinMartTheo}. It is an application of Dynkin's theorem and the specific knowledge of the generator \eqref{DefGen}. 
\end{proof}
We can use Taylor's theorem to expand $F( \langle \nu_v^{(K)}(s) \pm \frac{1}{K}\delta_{(x,z)}, \phiv \rangle)$ around $\langle \nu_v, \phi_v \rangle$ to obtain:
\begin{align}\label{Taylorexpan}
&\left( F_v(\nu_v^{(K)}(s) \pm \frac{1}{K}\delta_{(x,z)}) - F_v(\nu_v^{(K)}(s) ) \right) \nn \\
&\hspace{3cm}= \pm \frac{1}{K} \phi_v(x,z) F\myprime(\langle \nu_v^{(K)}(s), \phi_v\rangle) + \frac{1}{2K^2} \phi_v^2(x,z) F\mydprime(\langle \nu_v^{(K)}(s), \phi_v\rangle) + o(1/K^2),   
\end{align}
where $o(h^q)$ represents a function $G(h)$ satisfying 
\begin{equation*}
  \lim_{h \to 0 } \frac{G(h)}{h^q} = 0 
\end{equation*}
i.e., a function satisfying Peano's form of the remainder of Taylor's theorem.\\

From \eqref{Taylorexpan}, we can see that the martingale for the virus population becomes:
\begin{align*} 
 &\hspace{-2cm}F_v( \nu_v^{(K)}(t))  - \int_0^t A F_v(\nu_v^K(s),(\nuu^K(s),\nuc^K(s))) \,  ds + \epsilon_K^{F_v}(t)
\end{align*}
where
\begin{align*}
A &F_v(\nu_v,(\nuu,\nuc)) \nonumber \\
&= F\myprime(\langle \nu_v, \phi_v\rangle)  \int_{\Vp} b(x,z) \phi_v(x,z)  \nu_v(dx,dz) + \mu F\myprime(\langle \nu_v, \phi_v\rangle) \int_{\Vp} \phi_v(x,z) \left[ \phi_v(x,z) -\int_{ \X} m(z,e) \phi_v(x,e) \, de \right] \nu_v(dx,dz) \nonumber \\
&+ F\myprime(\langle \nu_v, \phi_v\rangle)   \int_{\Vp} \left( d(z) + c \, \langle \nu_v^{x}, 1\rangle \right) \phi_v(x,z)  \nu_v(dx,dz) -  F\myprime(\langle \nu_v, \phi_v\rangle) \int_{ E} \left(   \int_{\Vc } \eta(y,x,z) \phi_v(x,z)  \,  \nuc(dy,dz) \right)  dx  \nonumber \\
&+ F\myprime(\langle \nu_v, \phi_v\rangle) \int_{\Vp } \left( \int_{ \D} \beta(y,x,\langle (\nu_v^{(K)})^{x}, 1\rangle,z)\,  \nuu(dy) \right) \phi_v(x,z) \nu_v(dx,dz) 
\end{align*}
with domain:
\begin{equation}\label{DfDomainB}
    \D(A) = \lbrace F_{v} : F \in \C^2(\R),  \phi_v \text{ is measurable and bounded}  \rbrace \subset \C_b (\M_F(\Vp)),
\end{equation}
and where $\epsilon_K^{F_v}(t)$ satisfies \eqref{epsK}.

\begin{remark}
Notice that $\D(A)$ generates the set $\C_b (\M_F(\Vp))$.
\end{remark}
In a similar way we now present the martingale for the vector population. Notice that to apply Theorem \ref{AverageThm} and Example 2.3 from \cite{kurtz1992averaging}, we do not need to show that the the quadratic variation of the fast process vanishes as $K$ goes to infinity. We first introduce the following additional notation:
\begin{equation*}
  \Delta^{K,\lambda,c}_{y,z} F_{\text{vec}}( \nuu,\nuc) :=  F_{\text{vec}}( \nuu-\tfrac{1}{K^\lambda}\delta_{y} ,\nuc+\tfrac{1}{K^\lambda}\delta_{y,z})  -F_{\text{vec}}( \nuu ,\nuc)
\end{equation*}
and
\begin{equation*}
  \Delta^{K,\lambda,u}_{y,z} F_{\text{vec}}( \nuu,\nuc) :=  F_{\text{vec}}( \nuu+\tfrac{1}{K^\lambda}\delta_{y} ,\nuc-\tfrac{1}{K^\lambda}\delta_{y,z})  -F_{\text{vec}}( \nuu ,\nuc).
\end{equation*}
We then have the following proposition:
\begin{proposition}\label{lemmakmartReg20}
Under Assumption \ref{Assbetakreg2}, for the joint process $(\nuu^{(K)},\nuc^{(K)})$, we have:
\begin{align*}
M_K^{F_{\text{vec}}}(t) &= F_{\text{vec}}( \nuu^{(K)}(t) ,\nuc^{(K)}(t))  -K^{1-\lambda} \int_0^t \partial_u F_{\text{vec}}( \nuu^{(K)}(s) ,\nuc^{(K)}(s))\int_{\D} \L^u \phiu(y) \, \nuu^{(K)}(s)(dy) \, ds \nonumber \\
&\hspace{-1.3cm}-K^{1-2 \lambda} \int_0^t \partial_u^2 F_{\text{vec}}( \nuu^{(K)}(s) ,\nuc^{(K)}(s)) \int_{\D} \sigma^u(y)^2  \, |\nabla \phiu(y) |^2 \, \nuu^{(K)}(s)(dy) \, ds \nonumber \\
&\hspace{-1.3cm}-K^{1-\lambda} \int_0^t \partial_c F_{\text{vec}}( \nuu^{(K)}(s) ,\nuc^{(K)}(s)) \int_{\Vc} \L^c \phic(y,z) \, \nuc^{(K)}(s)(dy,dz) \, ds \nonumber \\
&\hspace{-1.3cm}-K^{1-2 \lambda} \int_0^t \partial_c^2 F_{\text{vec}}( \nuu^{(K)}(s) ,\nuc^{(K)}(s)) \int_{\Vc} \sigma^c(y)^2 | \nabla_y \phic(y,z) |^2 \, \nuc^{(K)}(s)(dy,dz) \, ds \nonumber \\
&\hspace{-1.3cm} -K \int_0^t  \int_{\Vp \times \D} \beta(s,y,x,\langle \nu_v^{x}, 1\rangle,z) \left( \Delta^{K,\lambda,c}_{y,z} F_{\text{vec}}( \nuu^{(K)}(s) ,\nuc^{(K)}(s)) \right) \nu_v^{(K)}(s)(dx,dz) \,  \nuu^{(K)}(s)(dy) \, ds \nonumber \\
&\hspace{-1.3cm} - K \int_0^t   \int_{\Vc \times E} \eta(s,y,x,z) \left( \Delta^{K,\lambda,u}_{y,z} F_{\text{vec}}( \nuu^{(K)}(s) ,\nuc^{(K)}(s)) \right) \nuc^{(K)}(s)(dy,dz) \,  dx \,   ds \nonumber \\
&\hspace{-1.3cm}-K \int_0^t \int_{\Vc} \gamma(z) \left( \Delta^{K,\lambda,c}_{y,z} F_{\text{vec}}( \nuu^{(K)}(s) ,\nuc^{(K)}(s)) \right) \, \nuc^{(K)}(s)(dy,dz) \, ds 
\end{align*}
is a càdlàg martingale. 
\end{proposition}
Using an analogous Taylor's expansion to that in  \eqref{Taylorexpan}, we reveal the following form for the martingale of the fast system: 
\begin{align*}
\hspace{-1cm}&F_{\text{vec}}( \nuu^{(K)}(t) ,\nuc^{(K)}(t))  - \int_0^t K^{1-\lambda}  B F_{\text{vec}}( \nuv^{(K)}(s),(\nuu^{(K)}(s) ,\nuc^{(K)}(s))) \,  ds +\delta_K^{F_{\text{vec}}}(t)
\end{align*}
where the operator $ B F_{\text{vec}}(\nu_v,(\nuu,\nuc))$ is given as follows: 
\begin{align}\label{DefBgen}
&B F_{\text{vec}}(\nu_v,(\nuu,\nuc))= \partial_u F_{\text{vec}}(\nuu,\nuc) \int_{\D} \L^u \phiu(y) \, \nuu(dy)  + \partial_c F_{\text{vec}}(\nuu,\nuc) \int_{\Vc} \L^c \phic(y,z) \, \nuc(dy,dz)  \nonumber \\
&-  \partial_u F_{\text{vec}}(\nuu,\nuc) \int_{\Vp \times \D} \beta(s,y,x,\langle \nu_v^{x}, 1\rangle,z) \phiu(y) \nuv(dx,dz) \,  \nuu(dy) \nn \\
&+ \partial_c F_{\text{vec}}(\nuu,\nuc) \int_{\Vp \times \D} \beta(s,y,x,\langle \nu_v^{x}, 1\rangle,z) \phic(y,z)  \nuv(dx,dz) \,  \nuu(dy \nonumber \\
&+ \partial_u F_{\text{vec}}(\nuu,\nuc) \int_{\Vc \times E} \eta(s,y,x,z) \phiu(y) \nuc(dy,dz) \,  dx -\partial_c F_{\text{vec}}(\nuu,\nuc) \int_{\Vc \times E} \eta(s,y,x,z) \phic(y,z)  \nuc(dy,dz) \,  dx  \nonumber \\
&+ \partial_u F_{\text{vec}}(\nuu,\nuc)\int_{\Vc} \gamma(z) \phiu(y)  \nuc(dy,dz) -\partial_c F_{\text{vec}}(\nuu,\nuc)\int_{\Vc} \gamma(z) \phic(y,z) \nuc(dy,dz) 
\end{align}
with domain:
\begin{equation*}
    \D(B) = \lbrace F_{\text{vec}} : F \in \C^2(\R),  \phiu \in \C^2_0(\D), \phic \in \C^{2,0}_0(\Vc)  \rbrace \subset \C_b(\M_F^\text{vec})
\end{equation*}
and 
\begin{align*}
\delta_K^{F_{\text{vec}}}(t)  &= -K^{1-2\lambda} \int_0^t \partial_u^2 F_{\text{vec}}( \nuu^{(K)}(s) ,\nuc^{(K)}(s)) \int_{\D} \sigma^u(y)^2  \, |\nabla \phiu(y) |^2 \, \nuu^{(K)}(s)(dy) \, ds \nonumber \\
&-K^{1-2 \lambda} \int_0^t \partial_c^2  F_{\text{vec}}( \nuu^{(K)}(s) ,\nuc^{(K)}(s)) \int_{\Vc} \sigma^c(y)^2 | \nabla_y \phic(y,z) |^2 \, \nuc^{(K)}(s)(dy,dz) \, ds + o(1/K^2)
\end{align*}
where with $o(1/K^2)$ we meant as in \eqref{Taylorexpan}.\\

Notice that, since $\lambda >0$, indeed we have:
\begin{equation*}
    \lim_{K \to \infty} K^{\lambda-1} \E \left[ \sup_{ t \in [0,T]}  \mid \delta_K^{F_{\text{vec}}}(t)  \mid  \right] = 0
\end{equation*}
for any $T>0$.\\

As a consequence we have the following martingale:
\begin{equation}\label{Bmartingale}
    \int_0^t \int_{\M_F} B F_{\text{vec}}(\xi_v(s), (\xiu,\xic)) \Gamma(ds \times d (\xiu,\xic) )
\end{equation}
where $\Gamma$ is as in Theorem \ref{AverageThm}.\\

Following Example 2.3 of \cite{kurtz1992averaging} we need to find a countable subset of $\D(B)$  such that \eqref{closurecondition} is satisfied, and as a consequence \eqref{Bmartingale} can be re-written as:
\begin{equation}\label{expPreCount}
     \int_0^t \left( \int_{\M_F} B F_{\text{vec}}(\xi_v(s), (\xiu,\xic)) \Pi_{\xi_v(s)} (d (\xiu,\xic) ) \right) \, ds = 0 
\end{equation}

\subsubsection*{Countable set}
In order to find the desired countable set that verifies \eqref{closurecondition} we proceed as follows:\\

First, notice that by the compactness of $\bar{\D}$ and $\X$ we know that the space $\M_F^\text{vec}$ is a locally compact separable and metrizable space (this can be found for example in Theorem 1.14 of \cite{li2010measure}). We denote by $ \closure[0]{\M_F^\text{vec}}$ the one point compactification of $\M_F^\text{vec}$. That is:
\begin{equation*}
    \closure[0]{\M_F^\text{vec}} = \M_F^\text{vec} \cup \lbrace \mu_\infty \rbrace
\end{equation*}
where we have extended the weak topology by imposing $\mu_n \to \mu_\infty$ if and only if $\mu_n(\bar{\D} \times \bar{\D} \times \X) \to \infty$.\\

Second, we have that $\C(\closure[0]{\M_F^\text{vec}})$ is Polish, and hence separable, with the uniform norm (see for example \cite{kechris2012classical}). As a consequence $\C_{b,0} (\M_F^\text{vec})$ ( bounded continuous  functions vanishing at $\mu_\infty$) is separable as well (because the one point compactification of $\M_F^\text{vec}$ is metrizable).\\

Third, we know that $\D(B)$ generates the whole  $\C_b(\M_F^\text{vec})$. As a consequence, its restriction to the set $\hat{\D}(B)$, generating only those functions vanishing at infinity is separable as well under the topology of uniform convergence.\\

Finally, by construction, the countable dense set $\hat{D}$ witnessing the separability of $\hat{\D(B)}$ is the set that verifies \eqref{closurecondition}. 

\subsubsection*{Conclusion from the averaging principle}
From \eqref{expPreCount} we conclude that 
\begin{equation*}
\int_{\M_F} B F_{\text{vec}}(\xi_v(t), (\xiu,\xic)) \Pi_{\xi_v(t)} (d (\xiu,\xic) )   = 0 
\end{equation*}
for all $t$, Lebesgue almost surely. By Example 2.3 from \cite{kurtz1992averaging}, we obtain that
\begin{align}\label{MartingaleA}
 &\hspace{-2cm}F_v( \xi_v(t))  - \int_0^t A F_v(\xi_v(s)) \,  ds 
\end{align}
is a Martingale, where
\begin{align*}
A &F_v(\xi_v) \nonumber \\
&= F\myprime(\langle \xi_v, \phi_v\rangle)  \int_{\Vp} b(x,z) \phiv(x,z)  \nuv(dx,dz) + \mu F\myprime(\langle \xi_v, \phi_v\rangle) \int_{\Vp} b(x,z) \left[ \phi_v(x,z) -\int_{ \X} m(z,e) \phi_v(x,e) \, de  \right] \nu_v(dx,dz) \nonumber \\
&+ F\myprime(\langle \xi_v, \phi_v \rangle)   \int_{\Vp} \left( d(z) + c \, \langle \nu_v^{x}, 1\rangle \right) \phi_v(x,z)  \nu_v(dx,dz) \nonumber \\
&+ F\myprime(\langle \xi_v, \phi_v\rangle) \int_{\Vp } \left(\int_{\M_F^\text{vec}} \int_{ \D} \beta(s,y,x,\langle \nu_v^{x}, 1\rangle,z)\,  \Pi_{\xi_v(s)}^B(\nuu(dy) \times \M_F(\Vc) ) \right) \phi_v(x,z) \nu_v(dx,dz)  \nonumber \\
&-  F\myprime(\langle \xi_v, \phi_v\rangle) \int_{ E} \left( \int_{\M_F^\text{vec}}  \int_{\Vc } \eta(s,y,x,z) \phi_v(x,z)  \,  \Pi_{\xi_v(s)}^B (\M_F(\D) \times \nuc(dy,dz)) \right)  dx  
\end{align*}
where $\Pi_{\xi_v(s)}^B$ is stationary for the generator \eqref{DefBgen}.

\begin{remark}
Notice that the presence of the function $F \in \C^2(\R)$ only through it first derivative implies that the process $\lbrace \xi_v \rbrace$ is deterministic, so that the martingale  in \eqref{MartingaleA} is equal to zero. As a consequence of this observation we can use the special case $F(x)=x$, which gives an expression that indeed corresponds to the RHS of \eqref{IDEvirusequal}.
\end{remark}

To conclude the characterization of the limit points for the case $\lambda < 1$, notice that the same arguments given for the case $\lambda=1$ in Section \ref{CharactLamdaequal} can be easily adapted to this case.

\section*{Acknowledgements}
M. Ayala is supported by the project ARCHIV of the Agence Nationale
de la Recherche (ANR-18-CE32-0004).

\printbibliography

\appendix
\section{Construction of the process}\label{SecWell}

Here we provide the technical details that guarantee the well-definedness of our process.  We first present the following proposition, which necessary to control the growth of the process and avoid explosions: 

\begin{proposition}\label{ControlJumpRate}
Under Assumption \ref{AssPherates} we have that there exists a positive constant $C$ such that for every measure $\nu \in \M_p := \M_p(E \times \X) \times \M_p(\D) \times \M_p(\D \times \X)$  the global jump rate, i.e. the rate at which a jump event takes place, is bounded by:
\begin{equation*}\label{CbndJump}
    C(T) \langle 1, \nu \rangle \left( 1 + \langle 1, \nu \rangle \right)
\end{equation*}
\end{proposition}

\begin{proof}
For a measure $\nuvec = (\nu_v, \nuu, \nuc) \in \M_p$, we denote by $R(\nuvec)$ its total jump rate. This rate is bounded from above by:
\begin{align}\label{DefRnu}
 R(\nuvec) &\leq  \int_{ E \times \X} b(x,z) \nu_v(dx,dz) + \int_{ E \times \X} \left[d(z) + c \int_{\X} \nu_v^{(x)}(de)\right] \nu_v(dx,dz) +  \int_{\Vp \times \D} \bar{\beta} \, \nuv(dx,dz) \,  \nuu(dy)  \nonumber \\
& + \int_{\Vc \times E} \bar{\eta} \, \nuc(dy,dz) \, \zeta_E( dx) + \int_{\Vc} \gamma(z)  \, \nuc(dy,dz),
\end{align}
We conclude using Assumption \ref{AssPherates} and  setting:
\[
C = \max \left\{ \bar{b}, \bar{d},\bar{\gamma}, \bar{\beta}, \bar{\eta}, c \right\}.
\]
\end{proof}

\subsection{Path-wise construction of the process }
In this section we rigorously define the Markov processes on path space $\mathcal{D}([0,\infty), \M_F)$,  with  generator $\L$ given by \eqref{DefGen}. Following \cite{meleard_stochastic_2015}, we provide an specific construction in terms of Poisson point measures.\\

Let us introduce some additional notation. Let $ \N^* = \N \setminus \lbrace 0 \rbrace $.
Let $ \lambda $ be the Lebesgue measure on $ \R_+ $, $ \lambda_c $ the counting measure on $ \N^* $ and $ \zeta_E $ the counting measure on $ E $. Moreover, for all $t\geq 0$, we introduce the following notation:
\begin{equation*}
\A(t) = [0,t] \times \N^* \times \R_+
\end{equation*}

\begin{definition}
Let $(\Omega, \f, \P)$ be a sufficiently large probability space. On this probability space we consider the following independent random elements:
\begin{itemize}
	\item $ Q_{inf} $ a Poisson random measure on $ \R_+ \times \N^* \times \N^* \times \R_+ $, with intensity $ \lambda(ds) \otimes \lambda_c(di) \otimes \lambda_c(dj) \otimes \lambda(d\theta) $,
	\item $ Q_{dis} $ a Poisson random measure on $ \R_+ \times E \times \N^* \times\R_+ $, with intensity $ \lambda(ds) \otimes \zeta_E(dx) \otimes \lambda_c(dj) \otimes \lambda(d\theta) $,
	\item $ Q_{los} $, $ Q_{d} $, $ Q_{cb} $ Poisson random measures on $ \R_+ \times \N^* \times \R_+ $, with intensity $ \lambda(ds) \otimes \lambda_c(di) \otimes \lambda(d\theta) $,
	\item $ Q_{bm} $, a Poisson random measure on $ \R_+ \times \N^* \times \mathcal{X} \times \lambda(d\theta) $, with intensity $ \lambda(ds) \otimes \lambda_c(di) \otimes \overline{m}(dz) \otimes \lambda(d\theta) $.
	\item  For $\alpha \in \lbrace u,c \rbrace$, the set $\lbrace W^{i,\alpha}; i \geq 1 \rbrace$ denotes a family of independent Brownian motions in $\R^d$.
\end{itemize}
Moreover, we enlarge the original probability space  $(\Omega, \f, \P)$ to the filtered probability space  $(\Omega, \f, \f_t,  \P)$ , where $( \f_t )_{t \geq 0}$ is the canonical filtration generated by $\lbrace Q_{inf}, Q_{dis}, Q_{los}, Q_{bm}  \rbrace$ and the families $\lbrace W^{i,\alpha}; i \geq 1 \rbrace$  for $\alpha \in \lbrace u,c \rbrace$.
\end{definition}

We then have the following representations:
\begin{itemize}
    \item Individuals (in the virus population) born from clonal births up to time $ t $ are given by
\begin{align*}
	\nuvec_{cb}(t) = \int_{\A(t)} \left(\delta_{x_i(s^-), z_i(s^-)},0,0\right) \, \indic{i \leq N_v(s^-)} \indic{\theta \leq (1-\mu) b(x_i(s^-), z_i(s^-))} Q_{cb}(ds\, di\, d\theta).
\end{align*}
    \item Individuals born with mutations are given by
\begin{align*}
	\nuvec_{bm}(t) = \int_{ \A(t) \times \X } \left(\delta_{x_i(s^-), z},0,0\right) \, \indic{i \leq N_v(s^-)} \indic{\theta \leq \mu b(x_i(s^-), z_i(s^-)) m(z_i(s^-), z)} Q_{bm}(ds\, di\, dz\, d\theta).
\end{align*}
    \item Individuals who died before time $ t \geq 0 $ are given by
\begin{align*}
	\nuvec_{d}(t) = \int_{\A(t)} \left(\delta_{x_i(s^-), z_i(s^-)},0,0 \right)\, \indic{i \leq N_v(s^-)} \indic{\theta \leq d(z_i(s^-)) + c N_{x_i(s^-)}(t)} Q_d(ds\, di\, d\theta).
\end{align*}
    \item Viruses being charged on a vector are represented by:
\begin{align*}
&\nuvec_{inf}(t)=  \int_{ \A(t) \times \N^*} \left(-\delta_{x_i(s^-),z_i(s^-)}, -\delta_{Y_j(s^-)}, \delta_{Y_j(s^-),z_i(s^-)}\right) \, \indic{i \leq N_v(s^-)} \indic{j \leq N_u(s^-)} \nn \\
&\times \indic{\theta \leq \beta(s^-, Y_j(s^-), x_i(s^-), N_{x_i(s^-)},z_i(s^-))} Q_{inf}(ds\, di\, dj\, d\theta).
\end{align*}
    \item Viruses who have been unloaded on a host plant up to time $ t $ are given by
\begin{align*}
	\nuvec_{dis}(t) = \int_{\A(t) \times E } \left(\delta_{x,z_j(s^-)}, \delta_{Y_j(s^-)}, -\delta_{Y_j(s^-),z_i(s^-)}\right)\, \indic{j \leq N_c(s^-)} \indic{\theta \leq \eta(s^-, Y_j(s^-), x, z_j(s^-))} Q_{dis}(ds\, dx\, dj\, d\theta).
\end{align*}
    \item Finally, viruses dying on vectors up to time t:
\begin{align*}
	\nuvec_{los}(t) = \int_{\A(t) } \left(0,\delta_{Y_i(s-)}, -\delta_{Y_i(s^-),z_i(s-)}\right)\, \indic{i \leq N_c(s^-)} \indic{\theta \leq \gamma(z_i(s^-))} Q_{los}(ds\, di\, d\theta).
\end{align*}
\end{itemize}

\begin{definition}\label{Definitionnuvec}
The process $ \nuvec(t) = (\nu_v(t), \nuu(t), \nuc(t); t \geq 0) $ is defined as the $\f_t$-adapted solution to the equation:
\begin{align}\label{EquYt}
\langle \phivec, \nuvec(t) \rangle &= \langle \phivec , \nuvec(0) \rangle + \langle \phivec, \nuvec_{cb}(t) \rangle +  \langle \phivec, \nuvec_{bm}(t) \rangle - \langle \phivec, \nuvec_{d}(t) \rangle + \langle \phivec, \nuvec_{inf}(t) \rangle +  \langle \phivec, \nuvec_{dis}(t) \rangle+ \langle \phivec, \nuvec_{los}(t) \rangle  \nonumber \\
&+ \int_{0}^{t} \langle \mathcal{L}^u \phivec, \nuvec(s) \rangle \, ds + \int_{0}^{t} \langle \mathcal{L}^c \phivec, \nuvec(s) \rangle \, ds + \int_{0}^{t} \sum_{i=1}^{N_u(s^-)} \sigma^u(Y^u_i(s^-)) \nabla \phi_{u}(Y^u_i(s^-)) \cdot dW^{i,u}_s \nonumber \\
&+ \int_{0}^{t} \sum_{i=1}^{N_c(s^-)} \sigma^c(Y^c_i(s^-)) \nabla_y\, \phic(Y^c_i(s^-), z_i(s^-)) \cdot dW^{i,c}_s,
\end{align}
for all  $\phivec \in \Phi(\M_p)$, and  $\L^\alpha \phivec$ is given as in Remark \ref{DiffuEnco}.
\end{definition}

\begin{remark}
The diffusion of vectors forces us to integrate the equation as observed in \cite{champagnat_invasion_2007} (see in particular equation 3.2, which contains analogous terms).
\end{remark}

The following proposition gives conditions to guarantee that a solution to \eqref{EquYt} follows the dynamics given by the generator $\L$ given in \eqref{DefGen}.

\begin{proposition}\label{PropGenerator}
Let $\left(\nuvec_t\right)_{t \geq 0} = (\nu_v(t), \nuu(t), \nuc(t) ; t \geq 0)$ be a solution of \eqref{EquYt} such that for all $T>0$ we have
\begin{equation*}
 \E \left( \sup_{t \leq T} < \mathbf{1}, \nuvec_t >^2 \right) < \infty,   
\end{equation*}
where $\mathbf{1} \in \Phi(\M_p)$ denotes the admissible triplet with all functions equal to the constant 1. Then, under Assumption \ref{AssPherates}, the process $\left(\nuvec_t\right)_{t \geq 0}$ is Markov with infinitesimal generator $\L$ given by \eqref{DefGen}.
\end{proposition}

\begin{proof}
The fact that it is Markov is standard. We need to verify that the generator is the one we claimed. We will use It\^o's lemma (see for example Theorem 5.1 in \cite{ikeda2014stochastic}) applied to \eqref{EquYt} to find an expression for $F_{\phivec}(\nuvec(t))$. Let us split this in two, one part coming from jump events, and the other from diffusion, as follows:
\begin{equation*}
 F_{\phivec} (\nuvec(t))-  F_{\phivec} (\nuvec(0)) =   (F_{\phivec} (\nuvec(t))-  F_{\phivec} (\nuvec(0)))_{\text{Jump}}
 + (F_{\phivec} (\nuvec(t))-  F_{\phivec} (\nuvec(0)))_{\text{Diff}} 
\end{equation*}
The  way to verify that this corresponds to the generator $\L$ is to take expectations and later differentiate with respect to time. We refer to \cite{meleard_stochastic_2015}, Proposition 6.3 in particular, for the same procedure in the absence of diffusion. Let us do this for the diffusive part of the generator. Notice that by It\^o's lemma, the diffusive part is given by:
\begin{align*}\label{DiffIto}
(F_{\phivec} (\nuvec(t)) &-  F_{\phivec} (\nuvec(0)))_{\text{Diff}} =  \int_{0}^{t} F_{\phivec}\myprime (\nuvec(s)) \langle \mathcal{L}^u \phivec, \nuvec(s) \rangle \, ds + \int_{0}^{t} F_{\phivec}\myprime (\nuvec(s)) \sum_{i=1}^{N_u(s^-)} \sigma^u(Y^u_i(s^-)) \nabla \phi_{u}(Y^u_i(s^-)) \cdot dW^{i,u}_s  \nonumber \\
&+ \frac{1}{2}\int_{0}^{t} F_{\phivec}\mydprime (\nuvec(s)) \sum_{i=1}^{N_u(s^-)} \sigma^u(Y^u_i(s^-))^2 \mid \nabla \phi_{u}(Y^u_i(s^-))\mid^2 \, ds+ \int_{0}^{t} F_{\phivec}\myprime (\nuvec(s))\langle \mathcal{L}^c \phivec, \nuvec(s) \rangle \, ds \nonumber \\
&+ \int_{0}^{t} F_{\phivec}\myprime (\nuvec(s)) \sum_{i=1}^{N_c(s^-)} \sigma^c(Y^c_i(s^-)) \nabla_y\, \phic(Y^c_i(s^-), z_i(s^-)) \cdot dW^{i,c}_s\nonumber \\
&+ \frac{1}{2}\int_{0}^{t} F_{\phivec}\mydprime (\nuvec(s)) \sum_{i=1}^{N_c(s^-)} \sigma^c(Y^c_i(s^-))^2 \mid \nabla_y\, \phic(Y^c_i(s^-), z_i(s^-)) \mid^2 \, ds.
\end{align*}
In order to verify that we get the second part of the generator $\L$, we have to proceed as before. Notice however, that by taking expectations the It\^o integrals vanish. Differentiating what is left at $t=0$ leads to the RHS of \eqref{DefgenL2}.
\end{proof}
Now we show the well-definedness of the process $( \nuvec_t)_{t \geq 0}$. That is, Theorem \ref{Theoremwelldefined}. Moreover we show that a control of the $p$-th moment at time zero can be extended to later times. More precisely, we show the following proposition:

\begin{proposition}\label{Propcontrolp}
Let $\nuvec_0 = (\nuv(0), \nuu(0),\nuc(0))$ be such that for some $p \geq 1$, we have:
\begin{equation}\label{pboundt0}
  \mathbb{E} \left(  \langle \mathbf{1}, \nuvec_0 \rangle^p  \right) < \infty      
\end{equation}
Then, under Assumption \ref{AssPherates}, the process $\left(\nuvec_t\right)_{t \geq 0} =\left( \nu_v(t),\nuu(t), \nuc(t) : t \geq 0  \right)$ satisfies:
\begin{equation}\label{pboundt}
  \mathbb{E} \left(  \sup_{t \in [0,T]}\langle \mathbf{1}, \nuvec_t \rangle^p  \right) < \infty      
\end{equation}
In particular, if \eqref{pboundt0} holds for $p=1$, we also have that the process $\nu_t$ is well defined.
\end{proposition}

\begin{proof}
To show \eqref{pboundt}, because the number of vectors is invariant under the dynamics, it is enough to show:
\begin{equation*}\label{pboundtvect}
  \mathbb{E} \left(  \sup_{t \in [0,T]}\langle 1, \nuv(t) \rangle^p  \right) < \infty.     
\end{equation*}
In order to do so, we use a stopping time argument. Let us define $\tau_n$ as follows:
\begin{equation*}
    \tau_n = \inf \left\{ t \geq 0: \langle 1, \nuv(t) \rangle \geq n \right\}
\end{equation*}
Let us also consider $F_p: \M_p^{\alpha} \to \R$ given by the choice:
\begin{equation*}
    F_p( \nuvec ) = \left( \langle 1,\nuv  \rangle\right)^p + \left( \langle 1,\nuc  \rangle\right)^p.
\end{equation*}
Hence, by \eqref{EquYt} we have:
\begin{align*}
\sup_{s \in [0,t \wedge \tau_n] } &\langle 1, \nuv(s) \rangle^p  \leq   \langle 1, \nuv(0)  \rangle^p  \nonumber \\
&\hspace{-0.8cm}+  \int_0^{t \wedge \tau_n} \int_{ \N^* \times \R_+} \left[ (\langle 1, \nuv(s-) \rangle +1 )^p  - \langle 1, \nuv(s-) \rangle^p \right]  \, \indic{i \leq N_v(s^-)} \indic{\theta \leq (1-\mu) b(x_i(s^-), z_i(s^-))} Q_{cb}(ds\, di\, d\theta) \nonumber \\
&\hspace{-0.8cm}+  \int_0^{t \wedge \tau_n} \int_{\N^* \times \R_+ \times E} \left[ (\langle 1, \nuv(s-) \rangle +1 )^p  - \langle 1, \nuv(s-) \rangle^p \right]  \, \indic{j \leq N_c(s^-)} \indic{\theta \leq \eta(s^-, Y_j(s^-), x, z_j(s^-))} Q_{dis}(ds\, dx\, dj\, d\theta) \nn \\
&\hspace{-0.8cm}+  \int_0^{t \wedge \tau_n} \int_{\N^* \times \mathcal{X} \times \R_+} \left[ (\langle 1, \nuv(s-) \rangle +1 )^p  - \langle 1, \nuv(s-) \rangle^p \right]  \indic{i \leq N_v(s^-)} \indic{\theta \leq \mu b(x_i(s^-), z_i(s^-)) m(z_i(s^-), z)} Q_{bm}(ds\, di\, dz\, d\theta).
\end{align*}
where the diffusion terms vanished due to the presences of derivatives of the constant function $1$, and we have dropped the integral terms with a negative contribution (death terms, and loading terms).\\

Taking expectations and using the bounds given by Assumption \ref{AssPherates} we obtain:
\begin{align*}
\E \left[ \sup_{s \in [0,t \wedge \tau_n] } \langle 1, \nuv(s) \rangle^p \right]  &\leq  \E \left[ \langle 1, \nuv(0)  \rangle^p \right] +  C \E \left[\int_0^{t \wedge \tau_n}  N_v(s) \left[ (\langle 1, \nuv(s) \rangle +1 )^p  - \langle 1,\nuv(s) \rangle^p \right]  \, \right] \, ds \nonumber \\
&+ C \E \left[ \int_0^{t \wedge \tau_n}  \left[ (\langle 1, \nuv(s) \rangle +1 )^p  - \langle 1, \nuv(s) \rangle^p \right] \right] \, ds. \nonumber\\
&\leq C \left( 1 + \E \left[\int_0^{t \wedge \tau_n}  N_v(s) \left[ (\langle 1, \nuv(s) \rangle +1 )^p  - \langle 1, \nuv(s) \rangle^p \right]  \, ds \right]  \right) 
\end{align*}
where the constant $C$ changed its value incorporating the constant given by \eqref{pboundt0}.\\

Using the fact that $N_v(s) \leq \langle 1, \nuv(s) \rangle$, and the simple inequality $(x +1)^p - x^p \leq C_p (1 + x^{p-1})$, we obtain:
\begin{align*}
\E \left[ \sup_{s \in [0,t \wedge \tau_n] } \langle 1, \nuv(s)\rangle^p \right] &\leq C_p \left( 1 + \E \left[\int_0^{t}   \left[ \langle 1, \nuv(s \wedge \tau_n) \rangle  + \langle 1, \nuv(s\wedge \tau_n) \rangle^p \right]  \, ds \right]  \right) \nn \\
&\leq C_p \left( 1 + \E \left[\int_0^{t}    \langle 1, \nuv(s \wedge \tau_n ) \rangle^p   \, ds \right]  \right)
\end{align*}
where again the constant $C_p$ changed its value incorporating new constants.\\

By Gronwalls inequality we then have:
\begin{equation}\label{GronwApp}
 \E \left[ \sup_{s \in [0,t \wedge \tau_n] } \langle 1, \nuv(s) \rangle^p \right] \leq   C_p e^{C_p t}
\end{equation}
for some constant $C_p$ independent of $n$.\\

From \eqref{GronwApp} we can deduce that $\tau_n$ goes to infinity as $n\to \infty $ a.s. We then apply Fatou's lemma to conclude:
\begin{equation*}
 \E \left[ \sup_{s \in [0,t ] } \langle 1, \nuv(s) \rangle^p \right] \leq \liminf_{n \to \infty} \E \left[ \sup_{s \in [0,t \wedge \tau_n] } \langle 1, \nuv(s) \rangle^p \right] \leq   C_p e^{C_p t}.
\end{equation*}

To conclude the well-definedness of the process $\nu_t$ one has to construct the process step by step, where the time steps are given by a sequence of jump instants $T_n$ exponentially distributed  with law:
\begin{equation*}
R(\nu_{n-1}) e^{-R(\nu_{n-1})t} 
\end{equation*}
and where the total jump rate $R(\nu)$ is given by \eqref{DefRnu}.\\

It is then enough to check that the sequence $T_n$ goes to infinity almost surely. This follows from
\begin{equation*}
  \E \left(  \sup_{t \in [0,T]}\langle 1, \nuvec_t \rangle  \right) < \infty      
\end{equation*}
which is a consequence of \eqref{pboundt} when $p=1$.
\end{proof}

\subsection{Relevant martingales}
Now we will introduce a few martingales that are relevant when computing scaling limits. Let us start from a simple application of Dynkin's theorem:

\begin{proposition}\label{DynkinMartTheo}
Let $\nu_0 = (\nuv(0), \nuu(0), \nuc(0)) $ be such that for some $p \geq 2$ we have:
\begin{equation*}
  \mathbb{E} \left(  \langle 1, \nuv_0 \rangle^p  \right) < \infty 
\end{equation*}
Let also $F \in \mathcal{C}^2(\R)$ and  $\phivec \in \Phi(\M_p)$, be such that  there exists a $C$, possibly dependent on $F$ and $\phivec$, such that:
\begin{equation}\label{ControlFplusgen}
|F_{\phivec} (\nu)| +  |\L F_{\phivec} (\nu)|  \leq C \,\left( 1+ \langle \nu, \mathbf{1} \rangle^p \right) 
\end{equation}
where $F_\phivec$ is given as in Definition \ref{CyliDef}\\

Then, under Assumption \ref{AssPherates}, we have that the process 
\begin{equation}\label{DynkinMart}
M_t(F_{\phivec}) = F_{\phivec} (\nuvec_t)) - F_{\phivec} (\nuvec_0)- \int_0^t \L F_{\phivec} (\nuvec_s) \, ds
\end{equation}
is a càdlàg martingale. 
\end{proposition}

\begin{proof}
From Proposition \ref{PropGenerator} and Dynkin's theorem we know that $M_t(F_{\alpha})$ is a local martingale. Hence, it is enough to show that the R.H.S. of \eqref{DynkinMart} is integrable. This is a consequence of Assumption  \eqref{ControlFplusgen} and Proposition \ref{Propcontrolp}.
\end{proof}

\begin{remark}\label{UsingDynkinMart}
Simple but tedious computations show that for, $1 \leq  q \leq p-1$,  $\phi_v \in \mathcal{C}_b^2( \Vp),\phiu \in \mathcal{C}_b^2(\D) \cap D(\L^{u})$ and $\phic \in \mathcal{C}_b^2( \Vc) \cap D(\L^c)$, the functions
\begin{align*}
  F_{v}^q(\nu_v)  :=  \left( \langle \phiv , \nuv \rangle \right)^q, \quad F_{u}^q(\nuu)  :=  \left(  \langle \phiu, \nuu \rangle \right)^q, \text{ and } F_{c}^q(\nuc)  :=  \left( \langle \phic, \nuc  \rangle \right)^q
\end{align*}
satisfy \eqref{ControlFplusgen}. 
\end{remark}

The following result is a consequence of Proposition \ref{DynkinMartTheo} and  Remark \ref{UsingDynkinMart}

\begin{proposition}\label{MartingalesProp}
Let $\phiv \in \mathcal{C}_b^2( \Vp),\phiu \in  D(\L^{u})$ and $\phic \in D(\L^{c})$, under Assumption \ref{AssPherates},
we have the following càdlàg martingales:
\begin{enumerate}
    \item For the virus population:
    \begin{align}\label{MartH}
        M_t^{\phi_v}  &= \langle \nuv(t) , \phiv\rangle- \langle \nuv(0) , \phiv\rangle - \int_0^t \int_{\Vp} b(x,z) \phiv(x,z) \nuv(s)(dx,dz) \,  ds \nonumber \\
        &+ \mu \int_0^t \int_{\Vp} b(x,z)\left[ \phiv(x,z) -\int_{ \X} m(z,z^\prime) \phiv(x,z^\prime) \, dz^\prime  \right] \nuv(s)(dx,dz) \,  ds \nonumber \\
        &+  \int_0^t \left[ \int_{\Vp} \left( d(z) + c \langle \nuv^{x}, 1\rangle \right) \phiv(x,z) \nuv(s)(dx,dz)  +   \int_{\Vp \times \D} \beta(s,y,x,\langle \nuv^{x}, 1\rangle,z) \phi_v(x,z) \nuv(s)(dx,dz) \,  \nuu(s)(dy) \right] ds \nonumber \\
        & -  \int_0^t   \int_{\Vc \times E} \eta(s,y,x,z) \phiv(x,z)  \,  \nuc(s)(dy,dz) \, dx \,   ds 
    \end{align}
    is a càdlàg martingale with predictable quadratic variation
    \begin{align}\label{MartHQV}
     \langle M^{\phi_v} \rangle_t &=  (1-\mu) \int_0^t \int_{\Vp} b(x,z) \, \phi_v(x,z)^2 \, \nu_v(s)(dx,dz) \,  ds + \mu \int_0^t \int_{\Vp \times \X} b(x,z) m(z,z^\prime)  \,  \phiv(x,z^\prime)^2  \, dz^\prime   \nuv(s)(dx,dz) \,  ds \nonumber \\
        &+  \int_0^t  \int_{\Vp} \left( d(z) + c \langle \nuv^{x}, 1\rangle \right)  \,   \phiv(x,z)^2 \, \nuv(s)(dx,dz) +  \int_0^t   \int_{\Vc \times E} \eta(s,y,x,z) \,  \phiv(x,z)^2   \,  \nuc(s)(dy,dz) \, dx \,   ds  \nn \\
        &+ \int_0^t  \int_{\Vp \times \D} \beta(s,y,x,\langle \nuv^{x}, 1\rangle,z)  \,   \phi_v(x,z)^2 \, \nuv(s)(dx,dz) \,  \nuu(s)(dy)  ds
    \end{align}
    \item For the free-vector population:
    \begin{align}\label{MartF}
        M_t^{\phiu}  &= \langle \nuu(t) , \phiu \rangle- \langle \nuu(0) , \phiu\rangle -\int_0^t \int_{\D} \L^u \phiu(y) \, \nuu(s)(dy) \, ds -  \int_0^t   \int_{\Vc \times E} \eta(s,y,x,z) \phiu(y) \nuc(s)(dy,dz) \,  dx \,   ds  \nonumber \\
        &+  \int_0^t \left[ \int_{\Vp \times \D} \beta(s,y,x,\langle \nuv^{x}, 1\rangle,z) \phiu(y) \nuv(s)(dx,dz) \,  \nuu(s)(dy) -\int_{\Vc} \gamma(z) \phiu(y) \, \nuc(s)(dy,dz)  \right] ds 
    \end{align}
    is a càdlàg martingale with predictable  quadratic variation
    \begin{align*}
     \langle M^{\phiu}  \rangle_t &=  \int_0^t \int_{\D} \sigma^u(y)^2  \, |\nabla \phiu(y) |^2 \, \nuu(s)(dy) \, ds +  \int_0^t  \int_{\Vc \times E} \eta(s,y,x,z) \phiu(y)^2 \nuc(s)(dy,dz) \,  dx \,   ds  \nonumber \\
        &+ \int_0^t  \int_{\Vp \times \D} \beta(s,y,x,\langle \nuv^{x}, 1\rangle,z) \,   \phiu(y)^2 \, \nuv(s)(dx,dz) \,  \nuu(s)(dy)  ds +\int_0^t \int_{\Vc} \gamma(z) \phiu(y)^2 \, \nuc(s)(dy,dz) \, ds  
    \end{align*}
    \item For the charged-vector population:
    \begin{align}\label{MartG}
        M_t^{\phic}  &= \langle \nuc(t) , \phic\rangle- \langle \nuc(0) , \phic\rangle -\int_0^t \int_{\Vc} \L^c \phic(y,z) \, \nuc(s)(dy,dz) \, ds +  \int_0^t   \int_{\Vc \times E} \eta(s,y,x,z) \phic(y,z) \nuc(s)(dw,du) \,  dx \,   ds \nonumber \\
        &-  \int_0^t \left[ \int_{\Vp \times \D} \beta(s,y,x,\langle \nuv^{x}, 1\rangle,z) \phic(y,z) \nuv(s)(dx,dz) \,  \nuu(s)(dy) -\int_{\Vc} \gamma(y) \phic(y,z) \, \nuc(s)(dy,dz)  \right] ds 
    \end{align}
    is a càdlàg martingale with quadratic variation
    \begin{align*}
        \langle M^{\phic} \rangle_t  &= \int_0^t \int_{\Vc} \sigma^c(y)^2 | \nabla_y \phic(y,z) |^2 \, \nuc(s)(dy,dz) \, ds +  \int_0^t   \int_{\Vc \times E} \eta(s,y,x,z) \phic(y,z)^2 \nuc(s)(dy,dz) \,  dx \,   ds \nonumber \\
        &+  \int_0^t \left[ \int_{\Vp \times \D} \beta(s,y,x,\langle \nuv^{x}, 1\rangle,z) \phic(y,z)^2 \nuv(s)(dx,dz) \,  \nuu(s)(dy) +\int_{\Vc} \gamma(y) \phic(y,z)^2 \, \nuc(s)(dy,dz)  \right] ds 
    \end{align*}
\end{enumerate}
\end{proposition}

\begin{proof}
Assume $p \geq 3$. By Theorem \ref{DynkinMartTheo} and Remark \ref{UsingDynkinMart} with $q=1$, we have that \eqref{MartH}, \eqref{MartF}, and \eqref{MartG} are càdlàg martingales. For the predictable quadratic variation we just use Remark \ref{UsingDynkinMart}  with  $q=2$, and It\^o formula toghether with the Doob-Meyer decomposition. See \cite{meleard_stochastic_2015} for examples on a similar context.
\end{proof}

\end{document}